\documentclass[psamsfonts]{amsart}
\usepackage{amsmath,amssymb,amsfonts}
\usepackage{eucal}
\usepackage{stmaryrd}
\usepackage[normalem]{ulem}
\usepackage{verbatim}

\usepackage{enumerate}
\newenvironment{enumroman}{\begin{enumerate}[\upshape (i)]}{\end{enumerate}}

\input xy
\xyoption{all}
\newdir{ >}{{}*!/-10pt/@{>}}

\title{Algebraic model structures}
\author{Emily Riehl}

\address{Department of Mathematics \\
University of Chicago \\
5734 S. University Ave.\\
Chicago, IL 60637}

\email{eriehl@math.uchicago.edu\\ www.math.uchicago.edu/$\sim$eriehl}

\thanks{Supported by a NSF Graduate Research Fellowship} 


\theoremstyle{plain}
\newtheorem{thm}{Theorem}[section] 
\newtheorem{cor}[thm]{Corollary}
\newtheorem{prop}[thm]{Proposition}
\newtheorem{lem}[thm]{Lemma}
\newtheorem*{thm1}{Theorem \ref{adjthm5}}

\theoremstyle{definition}
\newtheorem{defn}[thm]{Definition}
\newtheorem{ex}[thm]{Example}
\newtheorem{rmk}[thm]{Remark}

\newtheorem*{notation}{Notation}

\newtheorem*{thank}{Acknowledgments}
\newtheorem*{defn1}{Definition \ref{algquilldefn}}

\makeatletter
\let\c@equation\c@thm
\makeatother
\numberwithin{equation}{section}

\newcommand{\R}{\mathbb{R}}
\newcommand{\Q}{\mathbb{Q}}

\newcommand{\A}{\mathcal{A}}
\newcommand{\B}{\mathcal{B}}

\renewcommand{\L}{\mathbb{L}}
\newcommand{\C}{\mathcal{C}}
\newcommand{\F}{\mathcal{F}}
\newcommand{\G}{\mathcal{G}}
\newcommand{\W}{\mathcal{W}}
\newcommand{\M}{\mathcal{M}}
\newcommand{\K}{\mathcal{K}}
\newcommand{\I}{\mathcal{I}}
\newcommand{\J}{\mathcal{J}}

\newcommand{\cL}{\mathcal{L}}
\newcommand{\cR}{\mathcal{R}}

\newcommand{\CC}{\mathbb{C}}
\newcommand{\FF}{\mathbb{F}}

\newcommand{\Kl}{\mathbf{Kl}}
\newcommand{\coKl}{\mathbf{coKl}}

\renewcommand{\a}{\alpha}
\renewcommand{\b}{\beta}
\renewcommand{\d}{\delta}
\newcommand{\e}{\epsilon}
\newcommand{\g}{\gamma}
\renewcommand{\l}{\lambda}

\renewcommand{\t}{\tau}

\newcommand{\op}{\text{op}}

\newcommand{\id}{\text{id}}

\newcommand{\dom}{\text{dom}}
\newcommand{\cod}{\text{cod}}

\renewcommand{\lim}{\text{lim}}

\newcommand{\Hom}{\text{Hom}}
\newcommand{\Sq}{\text{Sq}}

\newcommand{\ladj}{\text{ladj}}

\newcommand{\Rtalg}{\mathbb{R}_t\text{-}\mathrm{{\bf alg}}}
\newcommand{\Ralg}{\mathbb{R}\text{-}\mathrm{{\bf alg}}}
\newcommand{\Ftalg}{\mathbb{F}_t\text{-}\mathrm{{\bf alg}}}
\newcommand{\Falg}{\mathbb{F}\text{-}\mathrm{{\bf alg}}}

\newcommand{\Qalg}{\mathbb{Q}\text{-}\mathrm{{\bf coalg}}}
\newcommand{\Ltalg}{\mathbb{L}_t\text{-}\mathrm{{\bf coalg}}}
\newcommand{\Lalg}{\mathbb{L}\text{-}\mathrm{{\bf coalg}}}
\newcommand{\Ctalg}{\mathbb{C}_t\text{-}\mathrm{{\bf coalg}}}
\newcommand{\Calg}{\mathbb{C}\text{-}\mathrm{{\bf coalg}}}

\newcommand{\ra}{\rightarrow}
\newcommand{\lra}{\longrightarrow}
\newcommand{\sr}{\stackrel}
\newcommand{\Ra}{\Rightarrow}
\newcommand{\ol}{\overline}

\newcommand{\Cat}{\text{\bf Cat}}
\newcommand{\Set}{\text{\bf Set}}
\newcommand{\Top}{\text{\bf Top}}
\newcommand{\sSet}{\text{\bf sSet}}
\newcommand{\Ch}{\text{\bf Ch}}
\newcommand{\CAT}{\text{\bf CAT}}
\newcommand{\FunF}{\text{\bf FF}}
\newcommand{\AWFS}{\text{\bf AWFS}}
\newcommand{\LAWFS}{\text{\bf LAWFS}}
\newcommand{\Cmd}{\text{\bf Cmd}}

\begin{document} 
 
\begin{abstract} 
We define a new notion of an algebraic model structure, in which the cofibrations and fibrations are retracts of coalgebras for comonads and algebras for monads, and prove ``algebraic'' analogs of classical results. Using a modified version of Quillen's small object argument, we show that every cofibrantly generated model structure in the usual sense underlies a cofibrantly generated algebraic model structure. We show how to pass a cofibrantly generated algebraic model structure across an adjunction, and we characterize the algebraic Quillen adjunction that results. We prove that pointwise algebraic weak factorization systems on diagram categories are cofibrantly generated if the original ones are, and we give an algebraic generalization of the projective model structure. Finally, we prove that certain fundamental comparison maps present in any cofibrantly generated model category are cofibrations when the cofibrations are monomorphisms, a conclusion that does not seem to be provable in the classical, non-algebraic, theory. 
\end{abstract}

\maketitle
\tableofcontents 

\section{Introduction}

Weak factorization systems are familiar in essence if not in name to algebraic topologists. Loosely, they consist of left and right classes of maps in a fixed category that satisfy a dual lifting property and are such that every arrow of the category can be factored as a left map followed by a right one. Neither these factorizations nor the lifts are unique; hence, the adjective ``weak.'' Two weak factorization systems are present in Quillen's definition of a model structure \cite{quillenhomotopical} on a category. Indeed, for any weak factorization system, the left class of maps behaves like the cofibrations familiar to topologists while the right class of maps behaves like the dual notion of fibrations.

Category theorists have studied weak factorization systems in their own right, often with other applications in mind. From a categorical point of view, weak factorization systems, even those whose factorizations are described \emph{functorially}, suffer from several defects, the most obvious of which is the failure of the left and right classes to be closed under all colimits and limits, respectively, in the arrow category.

\emph{Algebraic} weak factorization systems, originally called \emph{natural} weak factorization systems, were introduced in 2006 by Marco Grandis and Walter Tholen \cite{gtnatural} to provide a remedy. In an algebraic weak factorization system, the functorial factorizations are given by functors that underlie a comonad and a monad, respectively. The left class of maps consists of coalgebras for the comonad and the right class consists of algebras for the monad. The algebraic data accompanying the arrows in each class can be used to construct a canonical solution to any lifting problem that is natural with respect to maps of coalgebras and maps of algebras. A classical construction in the same vein is the path lifting functions which can be chosen to accompany any Hurewicz fibration of spaces \cite{mayclassifying}.

More recently, Richard Garner adapted Quillen's small object argument so that it produces algebraic weak factorization systems \cite{garnercofibrantly, garnerunderstanding}, while simultaneously simplifying the functorial factorizations so constructed. In practice, this means that whenever a model structure is cofibrantly generated, its weak factorization systems can be ``algebraicized'' to produce algebraic weak factorization systems, while the underlying model structure remains unchanged. 

The consequences of this possibility appear to have been thus far unexplored. This paper begins to do so, although the author hopes this will be the commencement, rather than the culmination, of an investigation into the application of algebraic weak factorization systems to model structures. At the moment, we do not have particular applications in mind to justify this extension of classical model category theory. However, these extensions feel correct from a categorical point of view, and we are confident that suitable applications will be found. 

Section \ref{backgroundsec} contains the necessary background. We review the definition of a weak factorization system and state precisely what we mean by a functorial factorization. We then introduce algebraic weak factorization systems and describe a few important properties. We explain what it means for a algebraic weak factorization system to be cofibrantly generated and prove a lemma about such factorization systems that will have many applications. More details about Garner's small object argument, including a comparison with Quillen's, are given later, as needed.  

Section \ref{algsec} is in many ways the heart of this paper. To begin, we define an \emph{algebraic model structure}, that is, a model structure built out of algebraic weak factorization systems instead of ordinary ones. One feature of this definition is that it includes a notion of a natural comparison map between the two functorial factorizations. As an application, one obtains a natural arrow comparing the two fibrant-cofibrant replacements of an object, which can be used to construct a category of algebraically bifibrant objects in our model structure. We prove that cofibrantly generated algebraic model structures can be passed across an adjunction, generalizing a result due to Daniel Kan.

The adjunction between the algebraic model structures in this situation has many interesting properties, consideration of which leads us to define an \emph{algebraic Quillen adjunction}. For such adjunctions, the right adjoint lifts to a functor between the categories of algebras for each pair of algebraic weak factorization systems, which should be thought of as an algebraization of the fact that Quillen right adjoints preserve fibrations and trivial fibrations. Furthermore, the lifts for the fibrations and trivial fibrations are natural, in the sense that they commute with the functors induced by the comparison maps. Dually, the left adjoint lifts to functor between the categories of coalgebras and these lifts are natural. In order to prove that the adjunction described above is an algebraic Quillen adjunction, we must develop a fair bit of theory, a task we defer to later sections.

In Section \ref{levelsec}, we describe the pointwise algebraic weak factorization system on a diagram category and prove that it is cofibrantly generated whenever the inducing one is. This result is only possible because Garner's small object argument allows the ``generators'' to be a category, rather than simply a set.  One place where such algebraic weak factorization systems appear is in Lack's trivial model structure on certain diagram 2-categories \cite{lackhomotopy}, and consequently, these algebraic model structures are cofibrantly generated in the new sense, but not in the classical one. We then use the pointwise algebraic weak factorization system together with the work of Section \ref{algsec} to obtain a generalization of the projective model structure on a diagram category.
 
In Section \ref{appsec}, we showcase some advantages of algebraicizing cofibrantly generated model structures. Using the characterization of cofibrations and fibrations as coalgebras and algebras, we have techniques for recognizing cofibrations constructed as colimits and fibrations constructed as limits that are not available otherwise. We use these techniques to prove the surprising fact that the natural comparison map between the algebraic weak factorization systems of a cofibrantly generated algebraic model category consists of pointwise cofibration coalgebras, at least when the cofibrations in the model structure are monomorphisms. We conclude by applying these techniques to prove that the fibrant replacement monad in this setting preserves certain trivial cofibrations, a fact relevant to the study of categories of algebraically fibrant objects, some of which can be given their own lifted algebraic model structure by recent work of Thomas Nikolaus \cite{nikolausalgebraic}. 

In Section \ref{adjsec}, we begin to develop the theory necessary to prove the existence of an important class of algebraic Quillen adjunctions. First, we describe what happens when we have an adjunction between categories with related algebraic weak factorization systems, such that the generators of the one are the image of the generators of the other under the left adjoint, a question that turns out to have a rather complicated answer. In this setting, the right adjoint lifts to a functor between the categories of algebras for the monads of the algebraic weak factorization systems and dually the left adjoint lifts to a functor between the categories of coalgebras, though the proof of this second fact is rather indirect. To provide appropriate context for understanding this result and as a first step towards its proof, we present three general categorical definitions describing comparisons between algebraic weak factorization systems on different categories. The first two definitions, of lax and colax morphisms of algebraic weak factorization systems, combine to give a definition of an adjunction of algebraic weak factorization systems, which is the most important of these notions. 

The most expeditious proofs of these results make use of the fact that the categories of algebras and coalgebras accompanying an algebraic weak factorization system each have a canonical composition law that is natural in a suitable double categorical sense; in particular each algebraic weak factorization system gives rise to two double categories, whose vertical morphisms are either algebras or coalgebras and whose squares are morphisms of such. This composition, introduced in Section \ref{backgroundsec}, provides a recognition principle that identifies an algebraic weak factorization system from either the category of algebras for the monad or the category of coalgebras for the comonad.  As a consequence, it suffices in many situations to consider either the comonad or the monad individually, which is particularly useful here. 

The existence of adjunctions of cofibrantly generated algebraic weak factorization systems demands an extension of the universal property of Garner's small object argument. We conclude Section \ref{adjsec} with a statement and proof of the appropriate change-of-base result, which we use to compare the outputs of the small object argument on categories related by adjunctions. This extension is not frivolous; a corollary provides exactly the result we need to prove the naturality statement in the main theorem of the final section.

In Section \ref{algquillensec}, we apply the results of the previous section to prove that there is a canonical algebraic Quillen adjunction between the algebraic model structures constructed at the end of Section \ref{algsec}. The data of this algebraization includes five instances of adjunctions between algebraic weak factorization systems. Two of these are given by the comparison maps for each algebraic model structure. The other three provide an algebraic description of the relationship between the various types of factorizations on the two categories.

For convenience, we'll abbreviate algebraic weak factorization system as \emph{awfs}, which will also be the abbreviation for the plural, with the correct interpretation clear from context. Similarly, we write \emph{wfs} for the singular or plural of weak factorization system. The wfs mentioned in this paper beyond Section \ref{wfsssec} underlie some awfs and are therefore functorial.

\begin{thank} The author would like to thank her advisors at Chicago and Cambridge -- Peter May and Martin Hyland -- the latter for introducing her to this topic and the former for many productive sessions discussing this work as well as very helpful feedback on innumerable drafts of this paper.  The author is also grateful for several conversations with Mike Shulman and Richard Garner, some of the results of which are contained in Theorem \ref{cofthm} and Lemma \ref{coflem}. The latter also conjectured Lemma \ref{laxawfslem}, which enabled a simplification of the initial proof of Theorem \ref{adjawfsthm}, while the former also commented on an earlier draft of this paper and suggested the definitions of Section \ref{adjsec} and the statement and proof of Corollary \ref{adjawfscor}. Anna Marie Bohmann suggested the notation for the natural transformations involved in an awfs. The author was supported by a NSF Graduate Research Fellowship.
\end{thank}

\section{Background and recent history}\label{backgroundsec}

There are many sources that describe the basic properties of weak factorization systems of various stripes (e.g., \cite{ktfactorization} or \cite{rtlax}). We choose not to give a full account here and only include the topics that are most essential.

First some notation. We write ${\bf n}$ for the category associated to the ordinal $n$ as a poset, i.e., the category with $n$ objects 0, 1, $\ldots$, $n-1$ and morphisms $i \ra j$ just when $i \leq j$. Let $d^0, d^1, d^2 \colon {\bf 2} \ra {\bf 3}$ denote the three functors which are injective on objects; the superscript indicates which object is not contained in the image. Precomposition induces functors $d_0, d_1, d_2 \colon \M^{\bf 3} \ra \M^{\bf 2}$ for any category $\M$; where we write $\M^{\A}$ for the category of functors $\A \ra \M$ and natural transformations. We refer to $d_1$ as the ``composition functor'' because it composes the two arrows in the image of the generating non-identity morphisms of ${\bf 3}$.

\begin{defn}\label{coddomdefn} We are particularly interested in the category $\M^{\bf 2}$, sometimes known as the \emph{arrow category} of $\M$. Its objects are arrows of $\M$, which we draw vertically, and its morphisms $(u,v) \colon f \Ra g$ are commutative squares\footnote{We depict morphisms of $\M^{\bf 2}$ with a double arrow because $(u,v)$ is secretly a natural transformation between the functors $f,g \colon {\bf 2} \ra \M$, though we do not often think of it as such.} $$\xymatrix{ \cdot \ar[r]^u \ar[d]_f & \cdot \ar[d]^g \\ \cdot \ar[r]_v & \cdot}$$ There are canonical forgetful functors $\dom,\cod \colon \M^{\bf 2} \ra \M$ that project to the top and bottom edges of this square, respectively.
\end{defn}

The material in Sections \ref{wfsssec} and \ref{ffssec} is well-known to category theorists at least, while the material in Sections \ref{awfsssec} - \ref{cgssec} is fairly new. Naturally, we spend more time in the latter sections than in the former.

\subsection{Weak factorization systems}\label{wfsssec}

Colloquially, a \emph{weak factorization system} consists of two classes of arrows, the ``left'' and the ``right'', that have a lifting property with respect to each other and satisfy a factorization axiom. The lifting property says that whenever we have a commutative square as in (\ref{lift}) with $l$ in the left class of arrows and $r$ in the right, there exists an arrow $w$ as indicated so that both triangles commute. 

\begin{notation} When every lifting problem of the form posed by the commutative square 
\begin{equation}\label{lift} \xymatrix{ \cdot \ar[d]_l \ar[r]^u & \cdot \ar[d]^r \\ \cdot \ar@{-->}[ur]_w \ar[r]_v & \cdot}
\end{equation} has a solution $w$, we write $l \boxslash r$ and say that $l$ has the \emph{left lifting property} (LLP) with respect to $r$ and, equivalently, that $r$ has the \emph{right lifting property} (RLP) with respect to $l$. If $\A$ is a class of arrows, we write $\A^{\boxslash}$ for the class of arrows with the RLP with respect to each arrow in $\A$. Similarly, we write ${}^{\boxslash}\A$ for the class of arrows with the LLP with respect to each arrow in $\A$. 

In general, $\A \subset {}^{\boxslash}\B$ if and only if $\B \subset \A^{\boxslash}$; in this situation, we write $\A \boxslash \B$ and say that $\A$ has the LLP with respect to $\B$ and, equivalently, that $\B$ has the RLP with respect to $\A$. The operations $(-)^{\boxslash}$ and ${}^{\boxslash}(-)$ form a Galois connection with respect to the posets of classes of arrows of a category, ordered by inclusion.
\end{notation}

\begin{defn} A \emph{weak factorization system} $(\mathcal{L},\mathcal{R})$ on a category $\M$ consists of classes of morphisms $\mathcal{L}$ and $\mathcal{R}$ such that \begin{enumroman} \item Every morphism $f $ in $\M$ factors as $r \cdot l$, with $l \in \mathcal{L}$ and $r \in \mathcal{R}$. \item $\mathcal{L} = {}^{\boxslash} \mathcal{R}$ and $\mathcal{R} = \mathcal{L}^{\boxslash}$. \end{enumroman}
\end{defn}

Any class $\cL$ that equals ${}^{\boxslash}\cR$ for some class $\cR$ is \emph{saturated}, which means that $\cL$ contains all isomorphisms and is closed under coproducts, pushouts, transfinite composition, and retracts. The class $\cR= \cL^{\boxslash}$ has dual closure properties, which again has nothing to do with the factorization axiom. The following alternative definition of a weak factorization system is equivalent to the one given above.

\begin{defn}\label{wfsdefn2} A \emph{weak factorization system} $(\mathcal{L},\mathcal{R})$ in a category $\M$ consists of classes of morphisms $\mathcal{L}$ and $\mathcal{R}$ such that \begin{enumroman} \item Every morphism $f $ in $\M$ factors as $r \cdot l$, with $l \in \mathcal{L}$ and $r \in \mathcal{R}$. \item $\mathcal{L} \boxslash \mathcal{R}$. \item $\cL$ and $\cR$ are closed under retracts. \end{enumroman}
\end{defn}

Any model structure provides two examples of weak factorization systems: one for the trivial cofibrations and the fibrations and another for the cofibrations and trivial fibrations. Indeed, a particularly concise definition of a model structure on a complete and cocomplete category $\M$ is the following: a \emph{model structure} consists of three class of maps $\C$, $\F$, $\W$ such that $\W$ satisfies the 2-of-3 property and such that $(\C \cap \W, \F)$ and $(\C, \W \cap \F)$ are wfs.\footnote{It is not immediately obvious that $\W$ must be closed under retracts but this does follow by a clever argument the author learned from Andr\'{e} Joyal \cite[\S F]{joyaltheoryI}.}

\subsection{Functorial factorization}\label{ffssec} 

\begin{defn}\label{ffdefn} A \emph{functorial factorization} is a functor $\vec{E}\colon \M^{\bf 2} \ra \M^{\bf 3}$ that is a section of the ``composition'' functor $d_1 \colon\M^{\bf 3} \ra \M^{\bf 2}$.
\end{defn}

Explicitly, a functorial factorization consists of a pair of functors $L, R \colon \M^{\bf 2} \ra\M^{\bf 2}$ such that $f = Rf \cdot Lf$ for all morphisms $f \in \M$ and such that the following three conditions hold: $$\cod L = \dom R, \hspace{.5cm} \dom L = \dom, \hspace{.5cm} \cod R = \cod.$$ Together, $L = d_2 \circ \vec{E}$ and $R = d_0 \circ \vec{E}$ contain all of the data of the functor $\vec{E}$. The fact that $L$ and $R$ arise in this way implies all of the conditions described above.

It will often be convenient to have notation for the functor $\M^{\bf 2} \ra \M$ that takes an arrow to the object it factors through, and we typically write $E$ for this, without the arrow decoration. With this notation, the functorial factorization $\vec{E}\colon \M^{\bf 2} \ra \M^{\bf 3}$ sends  a commutative square \begin{equation}\label{funfact} \raisebox{.25in}{\xymatrix{ \cdot \ar[r]^u \ar[d]_f & \cdot \ar[d]^g \\ \cdot \ar[r]_v & \cdot}}\ \text{to~a~commutative~rectangle}\ \raisebox{.52in}{\xymatrix{\cdot \ar[r]^u \ar[d]_{Lf} & \cdot \ar[d]^{Lg} \\ Ef \ar[r]^{E(u,v)} \ar[d]_{Rf} & Eg \ar[d]^{Rg} \\ \cdot \ar[r]_v & \cdot}}\end{equation} We'll refer to $E \colon \M^{\bf 2} \ra \M$ as the \emph{functor accompanying the functorial factorization} $(L,R)$.

\begin{defn}
A wfs is called \emph{functorial} if it has a functorial factorization with $Lf \in \cL$ and $Rf \in \cR$ for all $f$.  
\end{defn}

There is a stronger notion of wfs called an \emph{orthogonal factorization system}, abbreviated \emph{ofs}, in which solutions to a given lifting problem are required to be unique.\footnote{An example in $\Set$ takes the epimorphisms as the left class and the monomorphisms as the right class. When we exchange these classes the result is a wfs.} These are sometimes called \emph{factorization systems} in the literature. It follows from the uniqueness of the lifts that the factorizations of an ofs are always functorial. For this stronger notion, the left class is closed under all colimits and the right under all limits, taken in the arrow category. 

Relative to orthogonal factorization systems, wfs with functorial factorizations suffer from two principal defects. The first is that a functorial wfs on $\M$ does not induce a pointwise wfs on a diagram category $\M^{\A}$, where $\A$ is a small category. The functorial factorization does allow us to factor natural transformations pointwise, but in general the resulting left factors will not lift against the right ones, even though their constituent arrows satisfy the required lifting property. This is because the pointwise lifts which necessarily exist are not naturally chosen and so do not fit together to form a natural transformation.

The second defect is that the classes of a functorial wfs, as for a generic wfs, fail in general to be closed under all the limits and colimits that one might expect. Specifically, we might hope that the left class would be closed under all colimits in $\M^{\bf 2}$ and the right class would be closed under all limits. As those who are familiar with working with cofibrations know, this is not true in general. 

These failings motivated Grandis and Tholen to define \emph{algebraic weak factorization systems} \cite{gtnatural}, which are functorial wfs with extra structure that addresses both of these issues. 

\subsection{Algebraic weak factorization systems}\label{awfsssec} 

Any functorial factorization gives rise to two endofunctors $L, R \colon \M^{\bf 2} \ra \M^{\bf 2}$, which are equipped with natural transformations to and from the identity, respectively. Explicitly, $L$ is equipped with a natural transformation $\vec{\e} \colon L \Ra \id$ whose components consist of the squares $\vec{\e}_f = \raisebox{.25in}{\xymatrix{ \cdot \ar[d]_{Lf} \ar@{=}[r] & \cdot \ar[d]^f \\ \cdot \ar[r]_{Rf} & \cdot}}$. We call $\vec{\e}$ the \emph{counit} of the endofunctor $L$ and write $\e_f := Rf$ for the codomain part of the morphism $\vec{\e}_f$. Using the notation of Definition \ref{coddomdefn}, $\vec{\e} = (1,\e)$. The component $\e \colon E \Ra \cod$ is a natural transformation in its own right, where $E$ is as in (\ref{funfact}).

Dually, $R$ is equipped with a natural transformation $\vec{\eta} \colon \id \Ra R$ whose components are squares $\vec{\eta}_g = \raisebox{.25in}{\xymatrix{ \cdot \ar[d]_g \ar[r]^{Lg} & \cdot \ar[d]^{Rg} \\ \cdot \ar@{=}[r] & \cdot}}$. We call $\vec{\eta}$ the \emph{unit} of the endofunctor $R$ and write $\eta = \dom\, \vec{\eta}$ for the natural transformation $\dom \Ra E$.  We write $\vec{\eta} = (\eta,1)$ in the notation of Definition \ref{coddomdefn}. We call a functor $L$ equipped with a natural transformation to the identity functor \emph{left pointed} and a functor $R$ equipped with a natural transformation from the identity functor \emph{right pointed}, though the directional adjectives may be dropped when the direction (left vs.~right) is clear from context.

\begin{lem}\label{retcharlem} In a functorial wfs $(\cL,\cR)$, the maps in $\cR$ are precisely those arrows which admit an algebra structure for the pointed endofunctor $(R, \vec{\eta})$. Dually, the class $\cL$ consists of those maps that admit a coalgebra structure for $(L,\vec{\e})$.
\end{lem}
\begin{proof}
Algebras for a right pointed endofunctor are defined similarly to algebras for a monad, but in the absence of a multiplication natural transformation, the algebra structure maps need only satisfy a unit condition. If $g \in \cR$ then it lifts against its left factor as shown $\raisebox{.25in}{\xymatrix{ \cdot \ar@{=}[r] \ar[d]_{Lg} & \cdot \ar[d]^{g} \\ \cdot \ar@{-->}[ur]^t \ar[r]_{Rg} & \cdot}}$. The arrow $(t,1) \colon Rg \Ra g$ makes $g$ an algebra for $(R,\vec{\eta})$. Conversely, if $g$ has an algebra structure $(t,s)$ then the unit axiom implies that $\raisebox{.25in}{\xymatrix{\cdot \ar[d]_g \ar[r]^{Lg} & \cdot \ar[d]^{Rg} \ar[r]^t & \cdot \ar[d]^g \\ \cdot \ar@{=}[r] & \cdot \ar[r]_s & \cdot}}$ is a retract diagram (hence, $s=1$). Thus, $g$ is a retract of $Rg \in \cR$, which is closed under retracts. 
\end{proof}

The notion of a algebraic weak factorization system is an algebraization of the notion of a functorial wfs in which the above pointed endofunctors are replaced with a comonad and a monad respectively.

\begin{defn} An \emph{algebraic weak factorization system} (originally, \emph{natural weak factorization system}) on a category $\M$ consists a pair $(\L,\R)$, where $\L = (L,\vec{\e}, \vec{\d})$ is a comonad on $\M^{\bf 2}$ and $\R = (R, \vec{\eta}, \vec{\mu})$ is a monad on $\M^{\bf 2}$, such that $(L,\vec{\e})$ and $(R,\vec{\eta})$ are the pointed endofunctors of some functorial factorization $\vec{E}\colon \M^{\bf 2} \ra \M^{\bf 3}$.  Additionally, the accompanying natural transformation $\Delta \colon LR \Ra RL$ described below is required to be a distributive law of the comonad over the monad.
\end{defn}

Because the unit $\vec{\eta}$ arising from the functorial factorization necessarily has the form $\vec{\eta}_f = \raisebox{.25in}{\xymatrix{ \cdot \ar[d]_f \ar[r]^{Lf} & \cdot \ar[d]^{Rf} \\ \cdot \ar@{=}[r] & \cdot}}$, it follows from the monad axioms that $\vec{\mu}_f = \raisebox{.25in}{\xymatrix{ \cdot \ar[d]_{R^2f} \ar[r]^{\mu_f} & \cdot \ar[d]^{Rf} \\ \cdot \ar@{=}[r] & \cdot}}$ where $\mu \colon ER \Ra E$ is a natural transformation, with $E$ as in (\ref{funfact}).
Hence, $\R$ is a \emph{monad over} $\cod \colon \M^{\bf 2} \ra \M$, which means that $\cod R = \cod$, $\cod\, \vec{\eta} = \id_{\cod}$ and $\cod\, \vec{\mu} = \id_{\cod}$. This means that $Rf$ has the same codomain as $f$, and the codomain component of the natural transformations $\vec{\eta}$ and $\vec{\mu}$ is the identity.

Dually, $\L$ is a \emph{comonad over} $\dom$ (in the sense that it is a comonad in the 2-category $\CAT/\M$ on the object $\dom \colon \M^{\bf 2} \ra \M $). We write $\d \colon E \Ra EL$ for the natural transformation $\cod\, \vec{\d}$ analogous to $\mu = \dom\, \vec{\mu}$ defined above. As a consequence of the monad and comonad axioms, $\raisebox{.25in}{\xymatrix{ \cdot \ar[d]_{LRf} \ar[r]^{\d_f} & \cdot \ar[d]^{RLf} \\ \cdot \ar[r]_{\mu_f} & \cdot}}$ commutes for all $f$. (Indeed, the common diagonal composite is the identity.) These squares are the components of a natural transformation $\Delta \colon LR \Ra RL$, which is the distributive law mentioned above. In this context, the requirement that $\Delta$ be a distributive law of $L$ over $R$ reduces to a single condition: $\d \cdot \mu =  \mu_L \cdot E(\d, \mu) \cdot \d_R$.
Because the components of $\Delta$ are part of the data of $\L$ and $\R$, this distributive law does not provide any extra structure for the awfs; rather it is a property that we ask that the pair $(\L,\R)$ satisfy.\footnote{Grandis and Tholen's original definition did not include this condition, but Garner's does. Using Garner's definition, awfs are bialgebras with respect to a two-fold monoidal structure on the category of functorial factorizations (see \cite[\S 3.2]{garnercofibrantly}); the distributive law condition says exactly that the monoid and comonoid structures fit together to form a bialgebra. This category provides the setting for the proofs establishing the machinery of Garner's small object argument. We recommend that the first-time reader ignore these details; to repeat a quote the author has seen attributed to Frank Adams, ``to operate the machine, it is not necessary to raise the bonnet.''}

Given an awfs $(\L,\R)$, we refer to the $\L$-coalgebras as the left class and the $\R$-algebras as the right class of the awfs. Unraveling the definitions, an $\L$-coalgebra consists of a pair $(f,s)$, where $f$ is an arrow of $\M$ and $(1,s) \colon f \Ra Lf$ is an arrow in $\M^{\bf 2}$ satisfying the usual conditions so that this gives a coalgebra structure with respect to the comonad $\L$. The unit condition says that $s$ solves the canonical lifting problem of $f$ against $Rf$. Dually, an $\R$-algebra consists of a pair $(g,t)$ such that $g$ is an arrow of $\M$ and $(t,1) \colon Rg \Ra g$ is an arrow in $\M^{\bf 2}$, where $t$ lifts $Lg$ against $g$.

The algebra structure of an element $g$ of the right class of an awfs should be thought of as a chosen lifting of $g$ against any element of the left class. Given an $\L$-coalgebra $(f,s)$ and a lifting problem $(u,v) \colon f \Ra g$, the arrow $w= t \cdot E(u,v) \cdot s$  \begin{equation}\label{chosenlift}\xymatrix{ \cdot \ar[r]^u \ar[d]_{Lf} & \cdot \ar@<.5ex>[d]^{Lg}  \\ \cdot \ar@{-->}[r]^{E(u,v)} \ar@<-.5ex>[d]_{Rf} & \cdot \ar[d]^{Rg} \ar@<.5ex>@{-->}[u]^t \\ \cdot \ar@<-.5ex>@{-->}[u]_s \ar[r]_v & \cdot }\end{equation} is a solution to the lifting problem. In particular, all $\L$-coalgebras lift against all $\R$-algebras.

If we let $\cL$ and $\cR$ denote the arrows in $\M$ that have some $\L$-coalgebra structure or $\R$-algebra structure, respectively, then it is not quite true that $(\cL, \cR)$ is a wfs. This is because retracts of maps in $\cL$ will also lift against elements of $\cR$, but the categories of coalgebras for a comonad and algebras for a monad are not closed under retracts. We write $\overline{\cL}$ for the retract closure of $\cL$ and similarly for $\cR$ and refer to the wfs $(\overline{\cL}, \overline{\cR})$ as the \emph{underlying wfs} of $(\L,\R)$. It is, in particular, functorial.

\begin{rmk}\label{endofunctorrmk}
Because the class of $\L$-algebras is not closed under retracts, not every arrow in the left class of the underlying wfs $(\overline{\cL}, \overline{\cR})$ of the awfs $(\L,\R)$ will have an $\L$-coalgebra structure. The same is true for the right class. (But see Lemma \ref{retractlem}!)

However, as we saw in Lemma \ref{retcharlem}, every arrow of $\overline{\cL}$ will have a coalgebra structure for the left pointed endofunctor $(L,\vec{\e})$ and conversely every coalgebra will be an element of $\overline{\cL}$.  It follows that coalgebras for the pointed endofunctors underlying an awfs are closed under retracts; this can also be proved directly. In fact, the coalgebras for the pointed endofunctor underlying a comonad are the retract closure of the coalgebras for the comonad. The proof of this statement uses the fact that the map $(1,s) \colon f \Ra Lf$ makes $f$ a retract of its left factor $Lf$, which has a free coalgebra structure for the comonad $\L$. Similar results apply to the right class $\overline{\cR}$.
\end{rmk}

\begin{ex}\label{ofsex} Any orthogonal factorization system $(\cL,\cR)$ is an awfs. Orthogonal factorization systems are always functorial, with all possible choices of functorial factorizations canonically isomorphic. The comultiplication and multiplication natural transformations for the functors $L$ and $R$ are defined to be the unique solutions to the lifting problems $\raisebox{.25in}{\xymatrix{ \cdot \ar[d]_L \ar[r]^{LL} & \cdot \ar[d]^{RL} \\ \cdot \ar@{-->}[ur]_{\d} \ar@{=}[r] & \cdot}}$ and $\raisebox{.25in}{\xymatrix{ \cdot \ar[d]_{LR} \ar@{=}[r] & \cdot \ar[d]^R \\ \cdot \ar@{-->}[ur]^{\mu}\ar[r]_{RR} & \cdot}}$. Every element of $\cR$ has a unique $\R$-algebra structure and the structure map is an isomorphism. Similarly, every element of $\cL$ has a unique $\L$-coalgebra structure, with structure map an isomorphism. It follows that the classes of $\R$-algebras and $\L$-coalgebras are closed under retracts. The remaining details are left as an exercise.
\end{ex}

In light of Remark \ref{endofunctorrmk}, why does it make sense to use a definition of awfs that privileges coalgebra structures for the comonad $\L$ over coalgebras for the left pointed endofunctor $(L,\vec{\e})$, and similarly on the right? We suggest three justifications. The first is that coalgebras for the comonad are often ``nicer'' than coalgebras for the pointed endofunctor. In examples, the former are analogous to ``relative cell complexes'' while the latter are the ``retracts of relative cell complexes.'' A second reason is that we can compose coalgebras for the comonad in an awfs, meaning we can give the composite arrow a canonical coalgebra structure. This definition, which will be given in Section \ref{compssec}, uses the multiplication for the monad explicitly, so is not possible without this extra algebraic structure. Finally, and perhaps most importantly, coalgebras for a comonad are closed under colimits, as we will prove in Theorem \ref{closureprop}. There is no analogous result for $(L,\vec{\e})$-coalgebras. The upshot is that when examining colimits, the extra effort to check that a diagram lands in $\Lalg$ is often worth it.

\begin{rmk}\label{chosenrmk} The original name \emph{natural} weak factorization system is in some sense a misnomer. In most cases, the lift of a map $r$ in the right class against its left factor is not \emph{natural}; it's simply \emph{chosen} and recorded in the fact that associate to the arrow $r$ a piece of \emph{algebraic} data. Solutions to lifting problems of the form (\ref{lift}) are constructed by combining the coalgebraic and algebraic data of $l$ and $r$ with a functorial factorization of the square. These lifts are not natural with respect to all morphisms in the arrow category. They are however natural with respect to morphisms of $\Lalg$ and $\Ralg$, but that is true precisely because morphisms in a category of algebras are required to preserve the algebraic structure.

In an important special case, however, there are natural lifts; namely, for the free morphisms that arise as left and right factors of arrows. Hence, the adjective ``natural'' appropriately describes these factorizations. The multiplication of the monad $\R$ gives any arrow of the form $Rf$ a natural $\R$-algebra structure $\mu_f$. Similarly, the arrows $Lf$ have a natural $\L$-coalgebra structure $\d_f$ using the comultiplication of the comonad. Of course, it may be that there are other ways to choose lifting data for these arrows, but the natural choices provided by the comultiplication and multiplication have the property that the map from $Lf$ to $Lg$ or $Rf$ to $Rg$ arising from any map $(u,v) \colon f \Ra g$ preserves the lifting data. 
\end{rmk}

We conclude this section with one final definition that will prove very important in Section \ref{algsec} and beyond.

\begin{defn}\label{morawfsdefn} A \emph{morphism of awfs} $\xi \colon (\L,\R) \ra (\L',\R')$ is a natural transformation $\xi \colon E \Ra E'$ that is a \emph{morphism of functorial factorizations}, i.e., such that\begin{equation}\label{morff}\xymatrix{ & \cdot \ar[dl]_{L f} \ar[dr]^{L' f} & \\ Ef \ar[rr]^{\xi_f} \ar[dr]_{R f} && E'f \ar[dl]^{R' f} \\ & \cdot &}\end{equation} commutes, and such that the natural transformations $(1,\xi) \colon L \Ra L'$ and $(\xi,1) \colon R \Ra R'$ are comonad and monad morphisms, respectively, which means that these natural transformations satisfy unit and associativity conditions. It follows that a morphism of awfs $\xi$ induces functors $\xi_* \colon \Lalg \ra \L'\text{-}\mathrm{\bf coalg}$ and $\xi^* \colon\R'\text{-}\mathrm{\bf alg} \ra \Ralg$ between the Eilenberg-Moore categories of coalgebras and algebras.
\end{defn}

\subsection{Limit and colimit closure}\label{closuressec}

It remains to explain how an awfs rectifies the defects mentioned at the end of \ref{ffssec}. We will speak at length about induced pointwise awfs later in Section \ref{levelsec}, but we can deal with colimit and limit closure right now.

Let $\Ralg$ denote the Eilenberg-Moore category of algebras for the monad $\R$ and let $\Lalg$ similarly denote the category of coalgebras for $\L$. It is a well-known categorical fact that the forgetful functors $U \colon \Ralg \ra \M^{\bf 2}$, $U \colon \Lalg \ra \M^{\bf 2}$ create all limits and colimits, respectively, that exist in $\M^{\bf 2}$. It follows that the right and left classes of the awfs $(\L,\R)$ are closed under limits and colimits, respectively. We have proven the following result of \cite{gtnatural}.

\begin{prop}[Grandis-Tholen]\label{closureprop} If $\M$ has colimits (respectively limits) of a given type, then $\Lalg$ (respectively $\Ralg$) has them, formed as in $\M^{\bf 2}$.
\end{prop}

\begin{rmk}\label{closureproprmk} It is possible to interpret \ref{closureprop} too broadly. This does not say that for any diagram in $\M^{\bf 2}$ such that the objects have a coalgebra structure, the colimit will have a coalgebra structure. This conclusion will only follow if the maps of the colimit diagram are arrows in $\Lalg$ and not just in $\M^{\bf 2}$.

However, we do now have a method for proving that a particular colimit is a coalgebra: namely checking that the maps in the relevant colimiting diagram are maps of coalgebras. While this can be tedious, it will allow us to prove surprising results about cofibrations, which the author suspects are intractable by other methods. (See, e.g., Theorem \ref{cofthm}. It is also possible to prove Corollary \ref{Tthm} directly in this manner.)
\end{rmk}

\begin{ex} An example will illustrate this important point, though we have to jump ahead a bit. As a consequence of Garner's small object argument (see \ref{cgthm}), there is an awfs on $\Top$ such that the left class of its underlying wfs consists of the cofibrations for the Quillen model structure. It is well-known that the pushout of cofibrations is not always a cofibration. For example, the vertical maps of 
\begin{equation}\label{cofpushout}\xymatrix{ D^{n+1} \ar@{=}[d] & {*} \ar[d]^j \ar[l] \ar@{=}[r] & {*} \ar@{=}[d] \\ D^{n+1} & S^n \ar[r] \ar[l]_(.4){j_{n+1}} & {*}}\end{equation} are all cofibrations and coalgebras in the Quillen model structure,\footnote{The arrow $j$ inherits its cofibration structure as a pushout of the generating cofibration $j_n$ as shown $\raisebox{.25in}{\xymatrix{ S^{n-1} \ar[d]_{j_n} \ar@{}[dr]|(.8){\ulcorner}\ar[r]^u & {*} \ar[d]^j \\ D^n \ar[r]_v & S^n}}$. Explicitly, if $c_n \colon D^n \ra Qj_n$ gives $j_n$ its coalgebra structure, then the cone $(Cj, Q(u,v) \cdot c_n)$ gives $j$ its coalgebra structure, where $Q$ is the functor accompanying the functorial factorization of this awfs.\label{pushfootnote}} but the pushout $D^{n+1} \twoheadrightarrow S^{n+1}$ is not. This tells us that one of the squares of (\ref{cofpushout}) is not a map of coalgebras, and furthermore there are no coalgebra structures for the vertical arrows such that both squares are maps of coalgebras.

By contrast, the pushout of \begin{equation}\label{cofpushout2}\xymatrix{D^n \ar[d]_{i_N} & S^{n-1} \ar[d]^{j_n} \ar[r] \ar[l]_{j_n} & {*} \ar[d]^j \\ S^{n} & D^n \ar@{->>}[r] \ar[l]_{i_S} & S^n}\end{equation} is a cofibration and a coalgebra because all three vertical arrows have a coalgebra structure and the squares of (\ref{cofpushout2}) preserve them. (The maps $i_N, i_S \colon D^n \ra S^n$ include the disk as the northern or southern hemisphere of the sphere.) Of course, this fact could be deduced directly because the pushout $S^n \ra S^n \vee S^n$ is an inclusion of a sub-CW-complex, but in more complicated examples this technique for detecting cofibrations will prove useful.
\end{ex}

\subsection{Composing algebras and coalgebras}\label{compssec}

Unlike the situation for ordinary monads on arrow categories, the category of algebras for the monad of an awfs $(\L,\R)$ can be equipped with a canonical composition law, which is natural in a suitable ``double categorical'' sense, described below. Furthermore, the comultiplication for the comonad $\L$ can be recovered from this composition, so one can recognize an awfs by considering only the category $\Ralg$ together with its natural composition law. Later, in Section \ref{laxcolaxdefnssec}, we will extend this recognition principle to morphisms between awfs. In concrete applications, this allows us to ignore the category $\Lalg$, which we'll see can be a bit of a pain.

In this section, we give precise statements of these facts and describe their proofs. Their most explicit appearance in the literature is \cite[\S 2]{garnerhomomorphisms}, but see also \cite[\S A]{garnerunderstanding} or \cite[\S 6.3]{garnercofibrantly}. The dual statements also hold.

Recall that when $\R$ is a monad from an awfs $(\L,\R)$, an $\R$-algebra structure for an arrow $f$ has the form $(s,1) \colon Rf \Ra f$; accordingly, we write $(f,s)$ for the corresponding object of $\Ralg$. Let $(f,s), (f',s') \in \Ralg$. We say a morphism $(u,v) \colon f \Ra f'$ in $\M^{\bf 2}$ is a map of algebras (with the particular algebra structures $s$ and $s'$ already in mind) when $(u,v)$ lifts to a morphism $(u,v) \colon (f,s) \Ra (f',s')$ in $\Ralg$. It follows from the definition that this holds exactly when $s' \cdot E(u,v)= u \cdot s$, where $E \colon \M^{\bf 2} \ra \M$ is the functor accompanying the functorial factorization of $(\L,\R)$. This condition says that the top face of the following cube, which should be interpreted as a map from the algebra depicted on the left face to the algebra on the right face, commutes.
$${\small \xymatrix{ \cdot \ar[rrr]^{E(u,v)} \ar[dr]^s \ar[dd]_{Rf} & & & \cdot \ar[dr]^{s'} \ar'[d][dd]^{Rf'} \\ & \cdot \ar[rrr]^(.4)u \ar[dd]^(.3)f & & & \cdot \ar[dd]^{f'} \\ \cdot \ar@{=}[dr] \ar'[r][rrr]_v & & & \cdot \ar@{=}[dr] \\ & \cdot \ar[rrr]_v & & & \cdot}}$$

\begin{defn} Let $(f,s), (g,t) \in \Ralg$ with $\cod f = \dom g$. Then $gf$ has a canonical $\R$-algebra structure $$\xymatrix{E(gf) \ar[r]^-{\d_{gf}} & EL(gf) \ar[rr]^-{E(1, t\cdot E(f,1))} & & Ef \ar[r]^-s & \dom f}$$
where $\d \colon E \Ra EL$ is the natural transformation arising from the comultiplication of the comonad $\L$.
\end{defn}

Write $(g,t)\bullet(f,s) = (gf, t \bullet s)$ for this composition operation. It is natural in the following sense.

\begin{lem}\label{complem} Let $(u,v) \colon (f,s) \Ra (h,s')$ and $(v,w) \colon (g,t) \Ra (k,t')$ be morphisms in $\Ralg$. Then $(u,w) \colon (gf,t \bullet s) \Ra (kh, t' \bullet s')$ is a map of $\R$-algebras.
$$\xymatrix@C=20pt@R=20pt{\cdot \ar[r]^u \ar[d]_f & \cdot \ar[d]^h \\ \cdot \ar[r]_v \ar[d]_g & \cdot \ar[d]^k \\ \cdot \ar[r]_w & \cdot}$$
\end{lem}
\begin{proof}
The proof is an easy diagram chase.
\end{proof}

\begin{rmk}
It follows from Lemma \ref{complem} that algebras for a monad arising from an awfs $(\L,\R)$ form a (strict) double category {\bf Alg}$\R$: objects are objects of $\M$, horizontal arrows are morphisms in $\M$, vertical arrows are $\R$-algebras, and squares are morphisms of algebras. The content of Lemma \ref{complem} is that morphisms of algebras can be composed vertically as well as horizontally. It remains to check that composition of algebras is strictly associative, but this is a straightforward exercise.
\end{rmk}

Lemma \ref{complem} has a converse, which provides a means for recognizing awfs from categories of algebras.

\begin{thm}[Garner]\label{compcharthm} Suppose $\R$ is a monad on $\M^{\bf 2}$ over $\cod \colon \M^{\bf 2} \ra \M$. Specifying a natural composition law on $\Ralg$ is equivalent to specifying an awfs $(\L,\R)$ on $\M$.
\end{thm}
\begin{proof}
Because $\R$ is a monad over cod, the components of its unit define a functorial factorization on $\M$ (see the beginning of Section \ref{awfsssec}). In particular, the functor $L$ and counit $\vec{\e}$ have already been determined. It remains to define $\d \colon E \Ra EL$ so that $\vec{\d} = (1,\d) \colon L \Ra L^2$ makes $\L = (L,\vec{\e},\vec{\d})$ into a comonad satisfying the distributive law with respect to $\R$.

Given a natural composition law on the category of $\R$-algebras and a morphism $f \in \M$, we define $\d_f \colon Ef \ra ELf$ to be $$\d_f := \xymatrix@C=30pt{ Ef \ar[r]^-{E(L^2f,1)} & E(Rf \cdot RLf)  \ar[r]^-{\mu_f \bullet \mu_{Lf}} & ELf},$$ where $\mu_f \bullet \mu_{Lf}$ is the algebra structure for the composite of the free algebras $(RLf,\mu_{Lf})$ and $(Rf,\mu_f)$. Equivalently, $\d_f$ is defined to be the domain component of the adjunct to the morphism 
$$\xymatrix{ \cdot \ar[d]_f \ar[r]^{L^2f} & \cdot \ar[d]^{Rf \cdot RLf = U(Rf\cdot RLf, \mu_f \bullet \mu_{Lf})} \\ \cdot \ar@{=}[r] & \cdot}$$ 
with respect to the (monadic) adjunction $\xymatrix{ \Ralg \ar@<-1ex>[r]_-U \ar@{}[r]|-{\perp} & \M^{\bf 2} \ar@<-1ex>[l]_-F}.$ 

By taking adjuncts of the unit and associativity conditions for a comonad, it is easy to check that such $\d$ makes $\L$ a comonad. The distributive law can be verified using the fact that $\mu_f \bullet \mu_{Lf}$ is, as an algebra structure, compatible with the multiplication for the monad $\R$. We leave the verification of these diagram chases to the reader; see also \cite[Proposition 2.8]{garnerhomomorphisms}. 
\end{proof}

\subsection{Cofibrantly generated awfs}\label{cgssec}

There are a few naturally occurring examples of awfs where the familiar functorial factorizations for some wfs underlie a comonad and a monad. One toy example is the so-called ``graph'' factorization of an arrow through the product of its domain and codomain. There are more serious examples, including the wfs from the Quillen model structure on {\bf Ch}$_R$ and the folk model structure on {\bf Cat}. However, the examples topologists find in nature are less obviously ``algebraic,'' and consequently awfs have not generated a lot of interest among topologists. Recently, Garner has developed a variant of Quillen's small object argument, modeled upon a familiar transfinite construction from category theory, that produces \emph{cofibrantly generated} awfs. In any cocomplete category satisfying an appropriate smallness condition, general enough to include the desired examples, Garner's small object argument can be applied in place of Quillen's, and the resulting awfs have the same underlying wfs as those produced by the usual small object argument. The functorial factorizations are different but also arguably better than Quillen's in that the objects constructed are somehow ``smaller'' (in the sense that superfluous ``cells'' are not multiply attached) and also the transfinite process by which they are constructed actually converges, rather than terminating arbitrarily at some chosen ordinal. Furthermore, Garner's small object argument can be run for a generating small category, not merely for generating sets, a generalization whose power will become apparent in Section \ref{levelsec}.

In this section, we explain in detail the defining properties of \emph{cofibrantly generated} awfs, produced by Garner's small object argument. A more detailed overview of his construction is given in Section \ref{levelsec}, where it will first be needed. See also  \cite{garnercofibrantly} or \cite{garnerunderstanding}.

First, we extend the notation $(-)^{\boxslash}$ to categories over $\M^{\bf 2}$, as opposed to mere sets of arrows.

\begin{defn}\label{boxslashdefn} We define a pair of functors $$\xymatrix{ (-)^{\boxslash}\colon \mathrm{{\bf CAT}}/\M^{\bf 2} \ar@<1ex>[r] \ar@{}[r]|(.47){\perp} & (\mathrm{{\bf CAT}}/\M^{\bf 2})^{\op} \colon {}^{\boxslash}(-) \ar@<1ex>[l] }$$ that are mutually right adjoint. If $\J$ is a category over $\M^{\bf 2}$, the objects of $\J^{\boxslash}$ are pairs $(g,\phi)$, where $g$ is an arrow of $\M$ and $\phi$ is a \emph{lifting function} that assigns each square $\raisebox{.25in}{\xymatrix{\cdot \ar[d]_j \ar[r]^u & \cdot \ar[d]^g \\ \cdot \ar[r]_v & \cdot}}$ with $j \in \J$ a lift $\phi(j,u,v)$ that makes the usual triangles commute. We also require that $\phi$ is coherent with respect to morphisms in $\J$. Explicitly, given $(a,b) \colon j' \Ra j$ in $\J$, we require that $\phi(j', ua, vb) = \phi(j,u,v) \cdot b$, which says that the triangle of lifts in the diagram below commutes. $$\xymatrix{\cdot \ar[d]_{j'} \ar[r]^a & \cdot \ar[d]_(.35)j \ar[r]^u & \cdot \ar[d]^g \\ \cdot \ar[r]_b \ar@{-->}[urr] & \cdot \ar[r]_v \ar@{-->}[ur] & \cdot}$$

Morphisms $(g, \phi) \ra (g', \phi')$ of $\J^{\boxslash}$ are arrows in $\M^{\bf 2}$ that preserve the lifting functions. The category $\J^{\boxslash}$ is equipped with an obvious forgetful functor to $\M^{\bf 2}$ that ignores the lifting data.  When $\J$ is a set, the image of $\J^{\boxslash}$ under this forgetful functor is the set $\J^{\boxslash}$ defined in Section \ref{wfsssec}.
\end{defn}

Garner provides two definitions of a cofibrantly generated awfs \cite{garnerunderstanding}, though his terminology more closely parallels the theory of monads. An awfs $(\L,\R)$ is \emph{free} on a small category $J\colon \J \ra \M^{\bf 2}$ if there is a functor \begin{equation}\label{freediag}\xymatrix@!0@C=40pt{\J \ar[rr]^-{\lambda} \ar[dr]_J & & \Lalg \ar[dl]^{U} \\ & \M^{\bf^2} &}\end{equation} that is initial with respect to morphisms of awfs among functors from $\J$ to categories of coalgebras of awfs. A stronger notion is of an \emph{algebraically-free} awfs, for which we require that the composite functor \begin{equation}\label{algfreeiso}\Ralg \sr{\mathrm{lift}}{\lra} (\Lalg)^{\boxslash} \sr{\lambda^{\boxslash}}{\lra} \J^{\boxslash}\end{equation} is an isomorphism of categories. The functor ``lift'' uses the algebra and coalgebra structures of $\R$-algebras and $\L$-coalgebras to define lifting functions via the construction of \ref{chosenlift}. The isomorphism (\ref{algfreeiso}) should be compared with the isomorphism of sets $\cR \cong \J^{\boxslash}$, which is the usual notion of a cofibrantly generated wfs $(\cL, \cR)$.

We will say that the awfs produced by Garner's small object argument are \emph{cofibrantly generated}. Garner proves that these awfs are both free and algebraically-free; we will find occasion to use both defining properties. 

\begin{thm}[Garner]\label{cgthm} Let $\M$ be a cocomplete category satisfying either of the following conditions. \begin{itemize} \item[$(*)$] Every $X \in \M$ is $\a_X$-presentable for some regular cardinal $\a_X$.   \item[$(\dagger)$] Every $X \in \M$ is $\a_X$-bounded with respect to some proper, well-copowered  orthogonal factorization system on $\M$, for some regular cardinal $\a_X$. \end{itemize} Let $J \colon \J \ra \M^{\bf 2}$ be a category over $\M^{\bf 2}$, with $\J$ small. Then the free awfs on $\J$ exists and is algebraically-free on $\J$. 
\end{thm}

We won't define all these terms here. What's important is to know that the categories of interest satisfy one of these two conditions. Locally presentable categories, such as $\sSet$, satisfy $(*)$. {\bf Top}, {\bf Haus}, and {\bf TopGp}  all satisfy $(\dagger)$. We say a category $\M$ \emph{permits the small object argument} if it is cocomplete and satisfies either $(*)$ or $(\dagger)$.

\begin{rmk}This notion of cofibrantly generated is broader than the usual one --- see Example \ref{trivmodelstrex} for a concrete example --- as ordinary cofibrantly generated wfs are generated by a set of maps, rather than a category. We will refer to this as the ``discrete case'', discrete small categories being simply sets.
\end{rmk}

As is the case for ordinary wfs, cofibrantly generated awfs behave better than generic ones.  We conclude this introduction with an easy lemma, which will prove vital to proofs in later sections.

\begin{lem}\label{retractlem} If an awfs $(\L,\R)$ on $\M$ is cofibrantly generated, then the class $\mathcal{R}$ of arrows that admit an $\R$-algebra structure is closed under retracts.
\end{lem}
\begin{proof}
When $(\L,\R)$ is generated by $\J$, we have an isomorphism of categories $\Ralg \cong \J^{\boxslash}$ over $\M^{\bf 2}$. The forgetful functor $U\colon \Ralg \ra \M^{\bf 2}$ sends $(g, \phi) \in \J^{\boxslash}$ to $g$. We wish to show that its image is closed under retracts. Suppose $h$ is a retract of $g$ as shown $$\xymatrix{\cdot \ar[d]_h \ar[r]^{i_1} & \cdot \ar[d]_g \ar[r]^{r_1} & \cdot \ar[d]^h \\ \cdot \ar[r]_{i_2} & \cdot \ar[r]_{r_2} & \cdot}$$ Define a lifting function $\psi$ for $h$ by $$\psi(j,u,v) := r_1 \cdot \phi(j, i_1\cdot u, i_2\cdot v).$$ The equations from the retract diagram show that $\psi$ is indeed a lifting function. It remains to check that $\psi$ is coherent with respect to morphisms $(a,b) \colon j' \Ra j$ of $\J$. We compute $$\psi(j', u\cdot a, v\cdot b) = r_1 \cdot \phi(j', i_1\cdot u\cdot a, i_2\cdot v\cdot b) = r_1\cdot \phi(j, i_1\cdot u, i_2\cdot v) \cdot b = \psi(j, u,v) \cdot b,$$ as required.
\end{proof}

The upshot of Lemma \ref{retractlem} is that every arrow in the right class of the ordinary wfs $(\overline{\cL}, \overline{\cR})$ underlying a cofibrantly generated awfs $(\L,\R)$ has an $\R$-algebra structure. When our awfs is cofibrantly generated, we emphasize this result by writing $(\overline{\cL},\cR)$ for the underlying wfs. We will also refer to a lifting function $\phi$ associated to an element of $g \in \J^{\boxslash}$ as an algebra structure for $g$, in light of (\ref{algfreeiso}) and this result.

\begin{rmk} Garner proves the discrete version of Lemma \ref{retractlem} in \cite{garnerunderstanding}: when the generating category $\J$ is discrete, $\mathcal{R}$ is closed under retracts and the wfs $(\overline{\mathcal{L}}, \mathcal{R})$ is cofibrantly generated in the usual sense by this set of maps. As a consequence, the new notion of ``cofibrantly generated'' agrees with the usual one, in the case where they ought to overlap. 
\end{rmk}

As a final note, the composition law for the algebras of a cofibrantly generated awfs is particularly easy to describe using the isomorphism (\ref{algfreeiso}).

\begin{ex}\label{compex} Suppose $(\L,\R)$ is an awfs on $\M$ generated by a category $\J$. Suppose $(f,\phi), (g,\psi) \in \J^{\boxslash} \cong \Ralg$ are composable objects, i.e., suppose $\cod\, f = \dom\, g$. Their canonical composite is $(gf, \psi \bullet \phi)$ where $$\psi \bullet \phi(j,a,b) := \phi(j, a, \psi(j, f\cdot a, b)),$$ and this is natural in the sense described by Lemma \ref{complem}.
\end{ex}

In the remaining sections, we will present new results relating awfs to model structures, taking frequent advantage of the machinery provided by Garner's small object argument.

\section{Algebraic model structures}\label{algsec}

The reasons that most topologists care (or should care) about weak factorizations systems is because they figure prominently in model categories, which are equipped with an interacting pair of them. Using Garner's small object argument, whenever these wfs are cofibrantly generated, they can be algebraicized to produce awfs. This leads to the question: is there a good notion of an \emph{algebraic} model structure? What is the appropriate definition?

Historically, model categories arose to enable computations in the homotopy category defined for a pair $(\M, \W)$, where $\W$ is a class of arrows of $\M$ called the \emph{weak equivalences} that one would like to manipulate as if they were isomorphisms. But with all of the subsequent development of the theory of model categories, this philosophy that the weak equivalences should be of primary importance is occasionally lost. With this principle in mind, the author has decided that an algebraic model structure is something one should give a \emph{pair} $(\M,\W)$, rather than a category $\M$; that is to say, one ought to have a particular class of weak equivalences in mind already. This suggests the following ``minimalist'' definition. 

\begin{defn}\label{algmodelstr} An \emph{algebraic model structure} on a pair $(\M, \W)$, where $\M$ is a complete and cocomplete category and $\W$ is a class of morphisms satisfying the 2-of-3 property, consists of a pair of awfs $(\CC_t,\FF)$ and $(\CC,\FF_t)$ on $\M$ together with a morphism of awfs $$\xi \colon (\CC_t,\FF) \ra (\CC,\FF_t)$$ such that the underlying wfs of $(\CC_t,\FF)$ and $(\CC,\FF_t)$ give the trivial cofibrations, fibrations, cofibrations, and trivial fibrations, respectively, of a model structure on $\M$, with weak equivalences $\W$. We call $\xi$ the \emph{comparison map}.
\end{defn}

The comparison map $\xi$ gives an algebraic way to regard a trivial cofibration as a cofibration and a trivial fibration as a fibration. We will say considerably more about this in a moment. 

Let $\C_t$ denote the underlying class of maps with a $\CC_t$-coalgebra structure and define $\C$, $\F_t$, and $\F$ likewise. By definition $(\ol{\C_t}, \ol{\F})$ and $(\ol{\C}, \ol{\F_t})$ are the underlying wfs of $(\CC_t, \FF)$ and $(\CC,\FF_t)$, respectively, where the bar denotes retract closure. The triple $(\ol{C}, \ol{\F}, \W)$ arising from an algebraic model structure gives a model structure on $\M$ in the ordinary sense; we call this the \emph{underlying ordinary model structure} on $\M$.

We say that an algebraic model structure is \emph{cofibrantly generated} if the two awfs are cofibrantly generated, in the sense described in Section \ref{cgssec}. In this case, $\F = \overline{\F}$ and $\F_t = \overline{\F_t}$ by Lemma \ref{retractlem}.

It is convenient to have notation for the two functorial factorizations. Let $Q = \cod\, C = \dom\, F_t$ be the functor $\M^{\bf 2} \ra \M$ accompanying the functorial factorization of $(\CC,\FF_t)$, i.e., the functor that picks out the object that an arrow factors through. Let $R$ be the analogous functor for $(\CC_t, \FF)$. This notation is meant to suggest cofibrant and fibrant replacement, respectively.

With this notation, the comparison map $\xi\colon (\CC_t,\FF) \ra (\CC, \FF_t)$ consists of natural arrows $\xi_f$ for each $f \in \M^{\bf 2}$ such that \begin{equation}\label{compmap}\xymatrix{ & \dom\, f \ar[dl]_{C_t f} \ar[dr]^{C f} & \\ Rf \ar[rr]^{\xi_f} \ar[dr]_{F f} && Qf \ar[dl]^{F_t f} \\ & \cod\, f &}\end{equation} commutes. Because $\xi$ is a morphism of awfs, it induces functors $$\xi_* \colon \Ctalg \ra \Calg \quad \text{and} \quad \xi^*\colon \Ftalg \ra \Falg,$$ which provide an algebraic way to regard a trivial cofibration as cofibration and a trivial fibration as a fibration. These maps have the following property. Given a lifting problem between a trivial cofibration $j$ and a trivial fibration $q$, there are two natural ways to solve it: regard the trivial cofibration as a cofibration and use the awfs $(\CC,\FF_t)$ or regard the trivial fibration as a fibration and use the awfs $(\CC_t, \FF)$.\footnote{The map $\xi$ assigns $\CC$-coalgebra structures to $\CC_t$-coalgebras. Similarly, $\xi$ maps the trivial cofibrations which are merely coalgebras for the pointed endofunctor underlying $\CC_t$ to coalgebras for the pointed endofunctor of $\CC$, which we saw in Remark \ref{endofunctorrmk} suffices to construct lifts (\ref{chosenlift}), which have the naturality property of (\ref{twolifts}). Similar remarks apply to the fibrations.}
In figure (\ref{twolifts}) below, the former option solves the lifting problem $(u,v) \colon j \Ra q$ by tracing the path around the middle of the back of the cube, while the latter option traces along the front of the cube. Naturality of $\xi$ says that the lifts constructed by each method are the same!

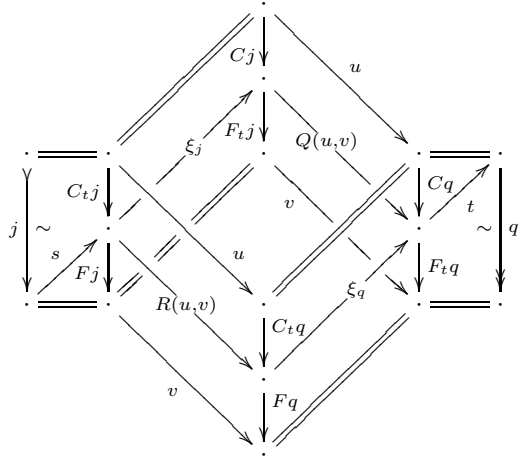
\begin{equation}\label{twolifts}\xymatrix@C=8pt@R=6pt{&& & & & & \cdot \ar[ddddrrrr]^u \ar[dd]_(.7){Cj} \\ {}\\
&&& & & & \cdot \ar'[dddrrr]|(.55){Q(u,v)}[ddddrrrr] \ar[dd]_(.7){F_tj} \\ {} \\
\cdot \ar@{=}[rr] \ar@{ >->}_j^{\sim}[dddd] &&\cdot \ar@{=}[uuuurrrr] \ar[ddddrrrr]^(.75)u \ar[dd]_{C_tj} & & & & \cdot \ar'[ddrr]_v'[dddrrr][ddddrrrr]  & & & & \cdot \ar[dd]^(.4){Cq} \ar@{=}[rr] & & \cdot \ar@{->>}^q_{\sim}[dddd] \\ & & & & & & &&& & \\
&&\cdot \ar'[ur][uuuurrrr]|(.4){\xi_j} \ar[ddddrrrr]|{R(u,v)} \ar[dd]_(.6){Fj}  & & & & & & & &  \cdot \ar[dd]^{F_tq} \ar[uurr]_t \\ & & & &&& & &&&  \\
\cdot \ar@{=}[rr] \ar[uurr]^s &&\cdot \ar@{=}'[ur]'[uurr][uuuurrrr] \ar[ddddrrrr]_v & & & & \cdot \ar@{=}[uuuurrrr]  \ar[dd]^(.3){C_tq} & & & & \cdot \ar@{=}[rr] & & \cdot \\ \\
&&&&&& \cdot \ar[uuuurrrr]|(.6){\xi_q} \ar[dd]^(.3){Fq}  \\ \\
&&& & & & \cdot \ar@{=}[uuuurrrr] }\end{equation}

\subsection{Comparing fibrant-cofibrant replacements}\label{fibcofsec} 

Any algebraic model structure induces a \emph{fibrant replacement monad} $\R$ and a \emph{cofibrant replacement comonad} $\Q$ on the category $\M$ (as opposed to the arrow category $\M^{\bf 2}$ on which the monads and comonads of the awfs act). The monad $\R$ arises as follows. The category $\M$ includes into $\M^{\bf 2}$ by sending an object $X$ to the unique arrow from $X$ to the terminal object. This inclusion is a section to the functor $\dom \colon \M^{\bf 2} \ra \M$. Because the monad $\FF$ is a monad over cod, it induces a monad $\R= (R, \eta, \mu)$ on $\M$ which we call the \emph{fibrant replacement monad}. The functor $R$ is obtained from the previous functor $R \colon \M^{\bf 2} \ra \M$ accompanying the functorial factorization of $(\CC_t,\FF)$ by precomposing $R$ by this inclusion.  We regret that our notation is somewhat ambiguous. The domain of $R$ should be apparent from whether an object in the image of $R$ is the image of an object or arrow of $\M$.  The arrows in the image of the two functors are related as follows: $Rf = R(f, 1_*)$, where $1_*$ denotes the identity at the terminal object. 

Dually, we can include $\M$ into $\M^{\bf 2}$ by slicing under the initial object. Using this inclusion, the comonad $\CC$ induces a comonad $\mathbb{Q} = (Q, \e, \d)$ on $\M$ which we call the \emph{cofibrant replacement comonad}. Once again, the functor $Q \colon \M \ra \M$ is obtained from the previous functor $Q\colon \M^{\bf 2} \ra \M$ by precomposing $Q$ by this inclusion.  Algebras for $\R$ are called \emph{algebraically fibrant objects} and coalgebras for $\Q$ are called \emph{algebraically cofibrant objects}. 

Another application of the natural lift illustrated in (\ref{twolifts}) is in comparing fibrant-cofibrant replacements of an object. Let $\M$ be a category with an algebraic model structure and let $X \in \M$. We can define its fibrant-cofibrant replacement to be either $RQX$ or $QRX$, both of which are weakly equivalent to $X$. Classically, there is no natural comparison between these choices, but in any algebraic model structure there is a natural arrow $RQX \ra QRX$ built out of the comparison map together with the components of the awfs.

\begin{lem}\label{fibcoflem} Let $\M$ be a category with an algebraic model structure and let $R$ and $Q$ be the induced fibrant and cofibrant replacement on $\M$. Then there is a canonical natural transformation $\chi \colon RQ \Ra QR$.
\end{lem}
\begin{proof}
Classically, one obtains a map $RQX \ra QRX$ by first lifting $i$ against $q$ and $j$ against $p$, as in the figure on the left below. Because the maps $i$, $j$, $p$, and $q$ are all obtained by factoring, they have free coalgebra or algebra structures for the awfs $(\CC_t, \FF)$ or $(\CC,\FF_t)$. Thus, each of these lifting problems has a natural solution (see Remark \ref{chosenrmk}). After a diagram chase, we can write the solution to the first lifting problem as $Q\eta_X$ and the second as $R\e_X$, using the unit and counit of the monad $\R$ and the comonad $\mathbb{Q}$.

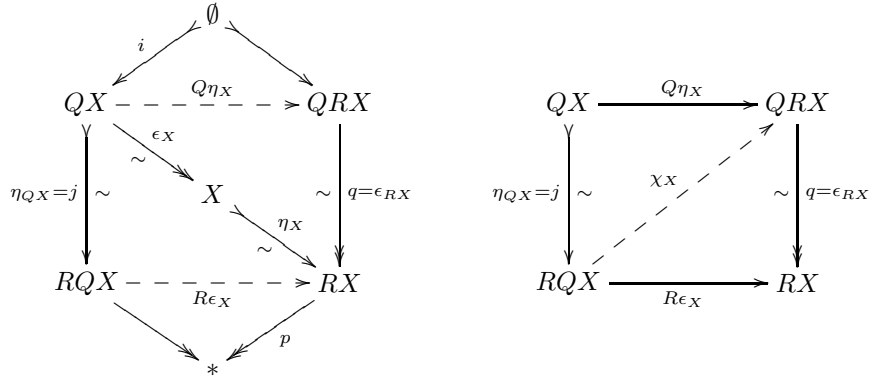
$$\xymatrix@R=20pt@C=25pt{ & \emptyset \ar@{ >->}[dl]_i \ar@{ >->}[dr] & \\ QX \ar@{ >->}[dd]_{\eta_{QX}=j}^{\sim} \ar@{->>}[dr]_{\sim}^{\e_X} \ar@{-->}[rr]^{Q\eta_X} & & QRX \ar@{->>}[dd]^{q=\e_{RX}}_{\sim} & & QX \ar@{ >->}[dd]^{\sim}_{\eta_{QX}=j} \ar[rr]^{Q\eta_X} & & QRX \ar@{->>}[dd]_{\sim}^{q=\e_{RX}} 
 \\ & X \ar@{ >->}[dr]_{\sim}^{\eta_X} & \\ RQX \ar@{->>}[dr] \ar@{-->}[rr]_{R \e_X} & & RX \ar@{->>}[dl]^p & & RQX \ar@{-->}[uurr]^{\chi_X} \ar[rr]_{R \e_X} & & RX
 \\ & {*} & }$$

The arrows $Q\eta_X$ and $R\e_X$ present a lifting problem between $j$ and $q$ that can be solved naturally using either awfs, as depicted in figure (\ref{twolifts}). The solutions to these lifting problems are the components of a natural transformation $RQ \Ra QR$ comparing the two fibrant-cofibrant replacements.
\end{proof}

This natural map is particularly well-behaved; hence the following theorem.

\begin{thm}\label{bifibthm}
The functor $Q$ lifts to a cofibrant replacement comonad on the category $\Ralg$ of algebraically fibrant objects. Dually, the functor $R$ lifts to a fibrant replacement monad on the category $\Qalg$ of algebraically cofibrant objects. The Eilenberg-Moore categories for this lifted comonad and lifted monad are isomorphic and give a notion of ``algebraically bifibrant objects.''
\end{thm}
\begin{proof} 
By a well-known categorical result \cite{powerwatanabecombining}, it suffices to find a natural transformation $\chi\colon RQ \Ra QR$ that is a distributive law of the monad over the comonad, i.e., a lax morphisms of monads and a colax morphism of comonads (see Section \ref{algssec}). The natural map of Lemma \ref{fibcoflem} satisfies the desired property: The defining lifting problem shows that $\chi$ is compatible with the unit and counit for $R$ and $Q$. It remains to show that $\chi$ is compatible with the multiplication, i.e., that 
$$\xymatrix@C=1pt{  & & RQR \ar[drr]^{\chi_R} & & \\ R^2Q \ar[urr]^{R\chi} \ar[dr]_{\mu_Q} & & & & QR^2 \ar[dl]^{Q\mu} \\ & RQ \ar[rr]^{\chi} & & QR &}$$ commutes, and a dual condition for the comultiplication. The necessary diagram chase uses the fact that the awfs $(\CC_t,\FF)$ satisfies the distributive law and the fact that the comparison map $\xi$ induces a monad morphism. We leave the remaining details to the reader.
\end{proof}

Unless we are talking about fibrant or cofibrant replacement specifically, $R$ and $Q$ will be functors $\M^{\bf 2} \ra \M$ accompanying the functorial factorizations of an algebraic model structure.

\subsection{The comparison map}\label{comparisonsec}  The least familiar component of the definition of an algebraic model structure given above is the comparison map. In figure (\ref{twolifts}), Lemma \ref{fibcoflem}, and Theorem \ref{bifibthm}, we saw some of its useful properties, but the question remains: in what circumstances might one expect a comparison map to exist? We discuss several answers to this question in this section.

\begin{rmk}\label{comprmk} Let $\J$ be the generating category for the awfs $(\CC_t,\FF)$ and let $(\CC,\FF_t)$ be an awfs on the same category. A comparison map between $(\CC_t, \FF)$ and $(\CC, \FF_t)$ exists if and only if there is a functor $\zeta \colon \J \ra \Calg$ over $\M^{\bf 2}$. This is because Garner's small object argument produces a canonical map $\lambda \colon \J \ra \Ctalg$ in {\bf CAT}/$\M^{\bf 2}$ that is universal among arrows from $\J$ to categories of coalgebras for the left half of an awfs on $\M$, in the sense that every such morphism $\zeta$ factors uniquely as $\xi_* \circ \lambda$, where $\xi$ is a morphism of awfs. See \cite[\S3]{garnerunderstanding}.
\end{rmk}

As far as the author is aware, model category theorists have not written about the issue of comparing the two wfs provided by an ordinary model structure, a fact that first came to her attention through discussions with Martin Hyland. But the existence of such a comparison map is more reasonable than one might expect: Peter May notes \cite{maypontoconcise2} that the universal property of the colimits in Quillen's small object argument gives such a natural transformation, provided we assume that the generating trivial cofibrations $\J$ are contained in the generating cofibrations $\I$.

In many cases, this admittedly untraditional assumption is quite reasonable: the generating trivial cofibrations are of course cofibrations, so including them with the generators does not change the resulting model structure. In the setting of algebraic model structures, this inclusion takes the form of a functor $\J \ra \I$ over $\M^{\bf 2}$, which induces a comparison map by the universal property of the ``free'' awfs generated by $\J$ (see Section \ref{cgssec}). When such a functor exists, ``the'' comparison map always refers to this one, though a priori some other might exist. In some of the results that follow, we require that there be a functor $\J \ra \I$ between the generating categories of a cofibrantly generated algebraic model structure without feeling too badly about it.

The following remark supports our Definition \ref{algmodelstr}.

\begin{rmk} Any ordinary cofibrantly generated model structure on a category that permits the small object argument can be made into an algebraic model structure by replacing the generating cofibrations $\I$ by $\I \cup \J$ and applying Garner's small object argument in place of Quillen's. The underlying ordinary model structure of the resulting algebraic model structure is the same as before, by which we mean that the classes $\overline{\C}$, $\F$, and $\W$ are unchanged. Thus, the abundance of cofibrantly generated model structures (in the ordinary sense) gives rise to an abundance of examples of algebraic model structures, which are then of course cofibrantly generated.
\end{rmk}

While altering the generating cofibrations does not change the underlying model structure, it does change the cofibration-trivial fibration factorization. Given that the generating cofibrations are often more natural than the generating trivial cofibrations,\footnote{Indeed, in many Bousfield localizations, the generating trivial cofibrations are not known explicitly.} we provide the following alternative method for obtaining a comparison map for a cofibrantly generated algebraic model structure by altering $\J$ as opposed to $\I$, vis-\`{a}-vis a theorem inspired by \cite[11.2.9]{hirschhornmodel}. As will become clear in the proof below, this method only applies in the case where the trivial cofibrations are generated by a set, as opposed to a category. 

In the following proof even though $\J$ is discrete, we regard it as a category over $\M^{\bf 2}$ with an injection $\J \sr{J}{\ra} \M^{\bf 2}$. With this perspective, we need a technical note. While Garner's small object argument works for any small category $J \colon \J \ra \M^{\bf 2}$ above the arrow category, in practice, the functor $J$ is injective, and we can identify $\J$ with its image and think of it as a set of arrows together with some coherence conditions in the form of morphisms between these arrows. As stated, Theorem \ref{icellprop} requires that $J$ be injective, though one could imagine that more careful wording would allow us to drop this assumption. We chose this simplification because we cannot think of any applications where this restriction is prohibitive.

\begin{thm}\label{icellprop} Suppose $\J$ is a set and $\I$ is a category over $\M^{\bf 2}$ such that the underlying wfs $(\overline{\C_t}, \F)$ and $(\overline{\C}, \F_t)$  of the awfs $(\CC_t,\FF)$ and $(\CC, \FF_t)$ that they generate give a model structure on $(\M, \W)$, in the ordinary sense. Then $\J$ can be replaced by a set $\J'$ over $\M^{\bf 2}$ such that \begin{enumroman} \item There is a functor $\J'^{\boxslash} \ra \J^{\boxslash}$ over $\M^{\bf 2}$, necessarily faithful, that is bijective on the underlying classes of arrows. It  follows that $\J'$ and $\J$ generate the same underlying wfs. \item There is a functor $\zeta\colon \J' \ra \Calg$ over $\M^{\bf 2}$. \end{enumroman} The set $\J'$ generates an awfs $(\CC_t', \FF')$.
It follows from the universal property of the functor $\J' \ra \mathbb{C}_t'\text{-}\mathrm{{\bf coalg}}$ that $\I$ and $\J'$ generate an algebraic model structure on $\M^{\bf 2}$ with the same underlying model structure $(\overline{\C},\F,\W)$.
\end{thm}
\begin{proof} Define $\J' \sr{J'}{\ra} \M^{\bf 2}$ to be the composite $\J \sr{J}{\ra} \M^{\bf 2} \sr{C}{\ra} \M^{\bf 2}$ where $C$ is the comonad generated by $\I$. For each $j \in \J$, the corresponding element of $\J'$ is its left factor $Cj$. We claim that $\J' = \{ Cj \mid j \in \J\}$ satisfies conditions (i) and (ii) above. For (ii), define $\zeta$ to be the map that assigns each $Cj$ its canonical free coalgebra structure $(Cj, \d_j)$. 

For (i), note that for each $j \in \J$ the lifting problem $\raisebox{.25in}{\xymatrix{ \cdot \ar[r]^{Cj} \ar[d]_j & \cdot \ar[d]^{F_t j} \\ \cdot \ar@{-->}[ur]^s \ar@{=}[r] & \cdot}}$ has a solution $s$, which gives us a retract diagram $\raisebox{.25in}{\xymatrix{ \cdot \ar@{=}[r] \ar[d]_j & \cdot \ar@{=}[r] \ar[d]^{Cj} & \cdot \ar[d]^j \\ \cdot \ar[r]_s & \cdot \ar[r]_{F_tj} & \cdot}}$. We use this to define a functor $\J'^{\boxslash} \ra \J^{\boxslash}$ as in the proof of Lemma \ref{retractlem}. On objects, define $$ \J'^{\boxslash} \ni (g, \psi) \mapsto (g, \phi) \in \J^{\boxslash},$$ where $\phi(j, u,v ) := \psi( Cj, u, v \cdot F_tj) \cdot s$ for all lifting problems $(u,v)\colon j \Ra g$. Because $\J$ is discrete, the lifting function $\phi$ need not satisfy any coherence conditions. Given a morphism $(h,k) \colon (g, \psi) \ra (g', \psi')$ in $\J'^{\boxslash}$, it follows that \begin{align*} \phi'(j, h\cdot u, k\cdot v) &= \psi'(Cj, h\cdot u, k\cdot v\cdot F_tj) \cdot s \\ &= h \cdot \psi (Cj, u, v \cdot F_t j) \cdot s \\ &= h \cdot \phi(j, u, v),\end{align*} which says precisely that $(h,k) \colon (g, \phi) \ra (g', \phi')$ is a morphism in $\J^{\boxslash}$. So $\J'^{\boxslash} \ra \J^{\boxslash}$ is a functor over $\M^{\bf 2}$.

It remains to show that this functor is surjective on the underlying arrows of $\J'^{\boxslash}$ and $\J^{\boxslash}$. Let $j \in \J$ and $(g, \phi) \in \J^{\boxslash}$; by definition $g \in \F$. By the 2-of-3 property, $Cj \in \C \cap \W \subset \ol{\C_t}$, so $Cj \boxslash g$. As $\J'$ is discrete, any choice of lifts against the $Cj$ can be used to define a lifting function $\psi$ so that $(g, \psi) \in \J'^{\boxslash}$. Of course the functor defined above need not map $(g,\psi)$ to $(g, \phi)$ but it does mean that $g$ is in the image when we forget down to $\M^{\bf 2}$, which is all that we claimed.
\end{proof}

\begin{rmk} We also have a faithful functor $\J^{\boxslash} \ra \J'^{\boxslash}$ over $\M^{\bf 2}$ that is bijective on the underlying classes of arrows; this one, however, takes a bit more effort to define. Define $ \J'' \sr{J''}{\ra} \M^{\bf 2}$ to be the composite $\J' \sr{J'}{\ra} \M^{\bf 2} \sr{C_t}{\ra} \M^{\bf 2}$. We have a functor $\l \colon \J'' \ra \Ctalg$ over $\M^{\bf 2}$ that assigns each arrow its free coalgebra structure. Mirroring the argument above, elements of $\J'$ are retracts of elements of $\J''$, so we have $\J''^{\boxslash} \ra \J'^{\boxslash}$ over $\M^{\bf 2}$. Our desired functor is the composite $$\J^{\boxslash} \cong \Falg \sr{\mathrm{lift}}{\lra} (\Ctalg)^{\boxslash} \sr{\l^{\boxslash}}{\lra} \J''^{\boxslash} \lra \J'^{\boxslash}$$ defined with help from the functor of \ref{boxslashdefn} and the isomorphism (\ref{algfreeiso}). Note that these functors are not inverse equivalences.
\end{rmk}

The upshot is that we can get an algebraic model structure from an ordinary cofibrantly generated model structure without changing the generating cofibrations. This argument does not appear to extend to non-discrete categories $\J$ because, in absence of a comparison map, the $\FF$-algebra structures of the $F_t j$ are \emph{chosen} and not \emph{natural} with respect to morphisms in $\J$; see Remark \ref{chosenrmk}. Note, the proof of Theorem \ref{icellprop} did not require that the awfs $(\CC,\FF_t)$ is cofibrantly generated, though in examples this is typically the case.

In Section \ref{appsec}, we will show that the components of the comparison map in a cofibrantly generated algebraic model structure satisfying additional, relatively mild, hypotheses are themselves $\CC$-coalgebras and hence cofibrations.

\subsection{Algebraic model structures and adjunctions} 

Many cofibrantly generated model structures are produced from previously known ones by passing the generating sets across an adjunction. We repeat this trick for cofibrantly generated algebraic model structures, extending a well-known theorem due to Kan \cite[11.3.2]{hirschhornmodel}.

An adjunction $\xymatrix{ T \colon \M \ar@<1ex>[r] \ar@{}[r]|-{\perp} & \K \colon S \ar@<1ex>[l]}$ lifts to an adjunction on the arrow categories $\M^{\bf 2}$ and $\K^{\bf 2}$, which we also denote by $T \dashv S$.  In particular, a small category $ J \colon \J \ra \M^{\bf 2}$ over $\M^{\bf 2}$ becomes a small category $TJ \colon \J \ra \K^{\bf 2}$ over $\K^{\bf 2}$. Because our notation has usually described the generating category as opposed to its functor to $\M^{\bf 2}$, we write $T\J$ to mean the category $\J$ that maps to $\K^{\bf 2}$ via the functor $TJ$.

\begin{thm}\label{adjthm} Let $\M$ have an algebraic model structure, generated by $\I$ and $\J$ and with weak equivalences $\W_{\M}$. Let $\xymatrix{ T \colon \M \ar@<1ex>[r] \ar@{}[r]|-{\perp} & \K \colon S \ar@<1ex>[l]}$ be an adjunction. Suppose $\K$ permits the small object argument and also that \begin{itemize} \item[$(\star)$]  $S$ maps arrows underlying the left class of the awfs cofibrantly generated by $T\J$ into $\W_{\M}$. \end{itemize} Then $T\J$ and $T\I$ generate an algebraic model structure on $\K$ with $\W_{\K} = S^{-1}(\W_{\M})$. Furthermore, $T \dashv S$ is a Quillen adjunction for the underlying ordinary model structures.
\end{thm}

In the literature, $(\star)$ is known as the \emph{acyclicity condition} because the arrows underlying the left class of the awfs generated by $T\J$ are the proposed \emph{acyclic} (trivial) cofibrations for the model structure on $\K$.

\begin{proof}[Proof of Theorem \ref{adjthm}]
By the small object argument, $T\J$ and $T\I$ generate awfs $(\CC_t,\FF)$ and $(\CC,\FF_t)$ with underlying wfs $(\ol{\C_t}, \F)$ and $(\ol{\C}, \F_t)$. With this notation, the condition $(\star)$ says that $S(\C_t) \subset \W_{\M}$. 

When the comparison map of the algebraic model structure on $\M$ arises from a functor $\J \ra \I$ over $\M^{\bf 2}$, composing with $T$ induces a functor $\J \ra \I$ over $\K^{\bf 2}$, which gives the comparison map between the resulting awfs. In the general case, the comparison map on $\M$ specifies a functor $\J \ra \mathbb{C}^{\M}\text{-}\mathrm{{\bf coalg}}$ to the category of coalgebras for the awfs on $\M$ generated by $\I$. In Corollary \ref{Tthm}, a significant result whose proof is deferred to Section \ref{adjsec}, we will prove that there is a functor $\mathbb{C}^{\M}\text{-}\mathrm{\bf coalg} \ra \Calg$ lifting $T$. This gives rise to a functor $\J \ra \Calg$ lifting $T$, or equivalently a functor $T\J \ra \Calg$ over $\K^{\bf 2}$. The comparison map $(\CC_t,\FF) \ra (\CC,\FF_t)$ for $\K$ is then induced by the universal property of the functor $T\J \ra \Ctalg$ produced by Garner's small object argument.

The class $\W_{\K}$ is retract closed by functoriality of $S$. It remains to show that $\overline{\C_t} = \overline{\C} \cap \W_{\K}$ and ${\F_t} = {\F} \cap \W_{\K}$. In fact, by \cite[11.3.1]{hirschhornmodel} we need only verify three of the four relevant inclusions.

The inclusion $\C_t \subset \C$ is immediate, since the comparison map explicitly provides each trivial cofibration with a cofibration structure; taking retract closures $\overline{\C_t} \subset \overline{\C}$. The hypothesis $(\star)$  says that $\C_t \subset \W_{\K}$ and $\W_{\K}$ is retract closed by functoriality of $S$, so $\overline{\C_t} \subset \W_{\K}$. Hence, $\overline{\C_t} \subset \overline{\C} \cap \W_{\K}$.

Similarly, the comparison map guarantees that $\F_t \subset \F$. If $g \in \F_t$ then it has some algebra structure $(g,\phi) \in T\I^{\boxslash}$ by Lemma \ref{retractlem}. By adjunction $(Sg, \phi^{\sharp}) \in \I^{\boxslash}$, where the arrows of $\phi^{\sharp}$ are the adjuncts of the corresponding arrows of $\phi$. So $Sg$ is a trivial fibration for the model structure on $\M$. In particular, $Sg \in \W_{\M}$, which says that $g \in \W_{\K}$. So $\F_t \subset \F \cap \W_{\K}$.

It remains to show that $\F \cap \W \subset {\F_t}$; we will appeal to Lemma \ref{retractlem} on two occasions. Suppose $f \in \F \cap \W_{\K}$. By Lemma \ref{retractlem}, $f$ has some algebra structure $(f,\psi) \in T\J^{\boxslash}$ and by adjunction $(Sf, \psi^{\sharp}) \in \J^{\boxslash}$. As $f \in \W_{\K}$, $Sf$ is a trivial fibration in the algebraic model structure on $\M$; by Lemma \ref{retractlem}, it follows that there is some algebra structure $\zeta$ such that $(Sf, \zeta) \in \I^{\boxslash}$.  By adjunction, $(f, \zeta^{\flat}) \in T\I^{\boxslash}$, where $\zeta^{\flat}$ denotes the adjunct of $\zeta$, which says that $f \in \F_t$, as desired.

The above argument showed that $S$ preserves fibrations and trivial fibrations. Hence, $T \dashv S$ is a Quillen adjunction.
\end{proof}

\subsection{Algebraic Quillen adjunctions}\label{algquillenssec}

Given the close connection between the algebraic model structures of Theorem \ref{adjthm}, it is not surprising that quite a lot more can be said about the nature of the Quillen adjunction between them. This leads to the notion of an \emph{algebraic Quillen adjunction}, of which the adjunction of Theorem \ref{adjthm} will be an example. We preview the definition and corresponding theorem below, but postpone the proofs, which are categorically intensive, to Sections \ref{adjsec} and \ref{algquillensec}. These sections are not dependent on the intermediate material, so a categorically inclined reader may wish to skip there directly.

Morphisms of awfs provide a means of comparing awfs on the same category, but as far as the author is aware, there are no such comparisons for awfs on different categories in the literature. We define three useful types of morphisms precisely in Section \ref{adjsec}, but here are the main ideas. 

Let $(\CC,\FF)$ and $(\L,\R)$ be awfs on $\M$ and $\K$ respectively. A \emph{colax morphism of awfs} $(\CC,\FF) \ra (\L,\R)$ is a functor $T\colon \M \ra \K$ together with a specified lifting of $T$ to a functor $\tilde{T} \colon \Calg \ra \Lalg$ satisfying one additional requirement. By a categorical result \cite{johnstoneadjoint}, the lift $\tilde{T}$ is determined by a characterizing natural transformation; together $T$ and this natural transformation is called a \emph{colax morphism of comonads} or simply a \emph{comonad morphism}. We ask that the natural transformation characterizing $\tilde{T}$ also determines an extension of $T$ to a functor $\hat{T} \colon \Kl(\FF) \ra \Kl(\R)$ between the Kleisli categories of the monads. 

The Kleisli category of a monad $\R$ is the full subcategory of $\Ralg$ on the free algebras, which are the objects in the image of the (monadic) free-forgetful adjunction. $\Kl(\R)$  is initial in the category of adjunctions determining that monad; the Eilenberg-Moore category $\Ralg$ is terminal. At the moment, the only justification we can give for this additional requirement, beyond the fact that it holds in important examples, is that colax morphisms of awfs should interact with both sides of the awfs. A more convincing justification is Lemma \ref{laxawfslem}.

Dually, a \emph{lax morphism of awfs} $(\L,\R) \ra (\CC,\FF)$ is a functor $S\colon \K \ra \M$ together with a specified lift $\tilde{S} \colon \Ralg \ra \Falg$ such that the natural transformation characterizing $\tilde{S}$ must determine an extension of $S$ to a functor between the coKleisli categories of $\L$ and $\CC$. When the functor $S$ or $T$ is the identity, both lax and colax morphisms of awfs are exactly morphisms of awfs, which is another clue that these are reasonable notions.  

Combining these, we arrive at the notion of \emph{adjunction of awfs}, which is the most relevant to this context. An \emph{adjunction of awfs} $(T,S) \colon (\CC,\FF) \rightarrow (\L,\R)$ consists of an adjoint pair of functors $\xymatrix{ T\colon \M \ar@<1ex>[r] \ar@{}[r]|-{\perp} & \K \colon S \ar@<1ex>[l]}$ such that $T \colon (\CC,\FF) \ra (\L,\R)$ is a colax morphism of awfs, $S \colon (\L,\R) \ra (\CC,\FF)$ is a lax morphism of awfs, and the characterizing natural transformations for these morphisms are related in a suitable fashion. 

Adjunctions of awfs over identity adjunctions are exactly morphisms of awfs, with both characterizing natural transformations equal to the natural transformation of Definition \ref{morawfsdefn}. Note, adjunctions of awfs can be canonically composed. We can now define \emph{algebraic Quillen adjunctions}.

\begin{defn}\label{algquilldefn} Let $\M$ have an algebraic model structure $\xi^{\M} \colon (\CC_t,\FF) \ra (\CC,\FF_t)$ and let $\K$ have an algebraic model structure $\xi^{\K} \colon (\L_t,\R) \ra (\L, \R_t)$. An adjunction $\xymatrix{ T \colon \M \ar@<1ex>[r] \ar@{}[r]|-{\perp} & \K \colon S \ar@<1ex>[l]}$ is an \emph{algebraic Quillen adjunction} if there exist adjunctions of awfs
$$\xymatrix@C=50pt{  (\CC_t, \FF) \ar[dr]|-{(T,S)} \ar[r]^{(T,S)} \ar[d]_{\xi^{\M}} & (\L_t,\R) \ar[d]^{\xi^{\K}} \\ (\CC,\FF_t) \ar[r]_{(T,S)}   & (\L,\R_t) }$$ such that both triangles commute.
\end{defn}

Note the left adjoint of an adjunction of awfs preserves coalgebras and hence (trivial) cofibrations and dually the right adjoint preserves algebras and hence (trivial) fibrations. In particular, an algebraic Quillen adjunction is a Quillen adjunction, in the ordinary sense.  As we shall prove in Section \ref{algquillensec}, the naturality condition of the definition of algebraic Quillen adjunction is equivalent to the condition that the lifts depicted below commute.
\begin{equation}\label{natliftdiag}{\small\xymatrix@R=15pt@C=5pt{ \Rtalg \ar[rr]^{\tilde{S}_t} \ar[dr]^{(\xi^{\K})^*}  \ar[dd] & & \Ftalg \ar[dr]^{(\xi^{\M})^*} \ar'[d][dd] 
 & & & \Ctalg \ar[dd] \ar[rr]^{\tilde{T}_t} \ar[dr]^{(\xi^{\M})_*} & & \Ltalg \ar'[d][dd] \ar[dr]^{(\xi^{\K})_*}
\\ & \Ralg \ar[dl] \ar[rr]^(.4){\tilde{S}} & & \Falg \ar[dl] & \text{and} & & **[l]\Calg \ar[dl] \ar[rr]^(.4){\tilde{T}} & & **[l]\Lalg \ar[dl] \\ \K^{\bf 2} \ar[rr]_S & & \M^{\bf 2} & & & \M^{\bf 2} \ar[rr]_T & & \K^{\bf 2} &}}\end{equation}
Similarly, the corresponding extensions to Kleisli and coKleisli categories commute. 

Somewhat surprisingly due to the numerous conditions required by their components, algebraic Quillen adjunctions exist in familiar situations.

\begin{thm}\label{adjthm5} 
Let $\xymatrix{ T \colon \M \ar@<1ex>[r] \ar@{}[r]|-{\perp} & \K \colon S \ar@<1ex>[l]}$ be an adjunction. Suppose $\M$ has an algebraic model structure, generated by $\I$ and $\J$, with comparison map $\xi^{\M}$. Suppose $\K$ has the algebraic model structure, generated by $T\I$ and $T\J$, with canonical comparison map $\xi^{\K}$. Then $T \dashv S$ is canonically an algebraic Quillen adjunction. 
\end{thm}

The proof is deferred to Section \ref{algquillensec}.

\section{Pointwise awfs and the projective model structure}\label{levelsec}

One of the features of an awfs that is not true of an ordinary wfs or even of a functorial wfs is that an awfs on a category $\M$ induces an awfs on the diagram category $\M^{\A}$ for any small category $\A$, where the factorizations are defined pointwise. The comultiplication and multiplication maps are precisely what is needed to define natural transformations that ensure that the left and right factors have the desired lifting properties. Furthermore, and completely unlike the non-algebraic situation, such \emph{pointwise awfs} are cofibrantly generated if the original awfs is. After proving this result, we will give an example of a class of cofibrantly generated algebraic model structures whose underlying ordinary model structures are not cofibrantly generated in the classical sense. 

We then construct a projective algebraic model structure on $\M^{\A}$ from a cofibrantly generated algebraic model structure on $\M$. The awfs in the projective model structure will not be the pointwise awfs, though these awfs will make an appearance in the proof establishing this model structure.

We first take a detour to describe Garner's small object argument in more detail, as these details will be used in the proofs in this section and the next.

\subsection{Garner's small object argument}\label{soassec}

Like Quillen's, Garner's small object argument produces a functorial factorization through a colimiting process that takes many steps, a key difference being that the resulting functorial factorization canonically underlies an awfs. Each step gives rise to a functorial factorization in which the left functor is a comonad. At the final step, the right functor is also a monad.

At step zero, Quillen's small object argument forms a coproduct over all squares from the generating cofibrations to the arrow $f$. In Garner's small object argument, this coproduct is replaced by a left Kan extension of the functor $J : \J \ra \M^{\bf 2}$ along itself. Write $L^0f = \mathrm{Lan}_J J(f)$ for the step zero comonad, called the density comonad in the literature. When $\J$ is discrete, $L^0f$ is the usual coproduct. In the general case, this arrow is a quotient of the usual coproduct.

The step-one factorization of both small object arguments is obtained the same way: by factoring the counit of the density comonad as a pushout followed by a square with an identity arrow on top. \begin{equation}\label{steponeeq}\xymatrix{ \cdot \ar[d]_{L^0f} \ar[r] \ar@{}[dr]|(.8){\ulcorner}& \cdot \ar@{=}[r] \ar[d]^{L^1f} & \cdot \ar[d]^f \\ \cdot \ar[r] & \cdot \ar[r]_{R^1f} & \cdot}\end{equation} Concretely, $L^1f$ is the pushout of $L^0f$ along the canonical arrow from the domain of $L^0f$ to the domain of $f$. The arrow $f$ and the arrow from the codomain of $L^0f$ to the codomain of $f$ form a cone under this pushout diagram; the unique map given by the universal property is $R^1f$. By the universal property of the pushout (\ref{steponeeq}), specifying an $(R^1,\vec{\eta}^1)$-algebra structure for $f$ is equivalent to specifying a lifting function $\phi$ such that $(f,\phi) \in \J^{\boxslash}$. 

The most significant difference between Garner's and Quillen's small object argument appears in the inductive steps that follow. For Quillen's small object argument the above processes are repeated with the arrow $R^{\a}f$ in place of $f$. We take the left Kan extension (coproduct over squares) and then pushout to obtain an arrow that is composed with the preceding left factors to obtain $L^{\a+1}f$. The map induced by the universal property is $R^{\a+1}f$. Transfinite composition is used to obtain $L^{\a}$ and $R^{\a}$ for limit ordinals. We choose to halt this process at some predetermined ``sufficiently large'' ordinal, yielding the final functorial factorization.

For Garner's small object argument, this process is modified to include additional quotienting. At step two and all subsequent steps, the beginning is the same. We pushout $L^0R^1f$ along the canonical arrow to obtain $L^1R^1f$. But then $L^2f$ is defined to be $L^1f$ composed with the coequalizer of two arrows from $L^1f$ to $L^1R^1f$. As in previous steps, this is a quotient of Quillen's definition. In the language of cell complexes, the arrow $L^1R^1f$ freely attaches new ``cells'' to the ``spheres'' in the domain of $R^1f$, while $L^1f$ includes those ``cells'' attached to ``spheres'' in the domain of $f$ into their image in the domain of $R^1f$. The coequalizer then avoids redundancy by identifying those ``cells'' attached to the same ``spheres'' in different stages. 

Unlike Quillen's small object argument, this quotienting means that when the category $\M$ permits the small object argument this process \emph{converges}; there is no need for an artificial termination point. The resulting object through which the arrow $f$ factors is in some sense ``smaller'' than for the factorizations produced by Quillen's small object argument because cells are attached only once, not repeatedly. The monad $\R$ is algebraically-free on the pointed endofunctor $R^1$, which says that $\Ralg \cong (R^1,\vec{\eta}^1)$-alg $\cong \J^{\boxslash}$. 

\subsection{Pointwise algebraic weak factorization systems} 

We now turn our attention to pointwise awfs. Because {\bf CAT} is cartesian closed, we have isomorphisms $(\M^{\A})^{\bf 2} \cong \M^{\A \times {\bf 2}} \cong (\M^{\bf 2})^{\A}$, which we use to regard a natural transformation $\a$ as a functor $\a \colon \A \ra \M^{\bf 2}$. On objects, this functor picks out the constituent morphisms of $\a$; the image of a morphism in $\A$ is the corresponding naturality square. Morphisms $(\phi,\psi):\a \Ra \b$ in the category of functors $\A \ra \M^{\bf 2}$ consist of a pair of morphisms $\phi, \psi$ in $\M^{\A}$ such that the vertical composites $\b \phi$ and $\psi \a$ are equal.

Given an awfs $(\L,\R)$ on $\M$, we use these isomorphisms to define $L^{\A}$ to be the functor $$(\M^{\A})^{\bf 2} \cong (\M^{\bf 2})^{\A} \sr{ (L)_*}{\lra} (\M^{\bf 2})^{\A} \cong (\M^{\A})^{\bf 2}$$ induced by post-composition by $L$; similarly for $R^{\A}$. We define the natural transformations $\vec{\e}^{\A}, \vec{\d}^{\A}, \vec{\eta}^{\A}, \vec{\mu}^{\A}$ that make $L^{\A}$ and $R^{\A}$ into a comonad and monad as follows. Given an object $\a$ of $(\M^{\A})^{\bf 2}$ regarded as a functor $\A\ra \M^{\bf 2}$, the arrow $\vec{\e}^{\A}_{\a}$ is obtained by ``whiskering'' $\vec{\e}$ with $\a$, as depicted below.
$$\xy 
(-24,0)*+{\A}="2";
(-8,0)*+{\M^{\bf 2}}="4"; 
(8,0)*+{\M^{\bf 2}}="6"; 
{\ar^-{\a} "2";"4"};
{\ar@/^1.65pc/^{L} "4";"6"}; 
{\ar@{=}@/_1.65pc/ "4";"6"}; 
{\ar@{=>}^<<<{\scriptstyle \vec{\e}} (0,3)*{};(0,-3)*{}} ; 
\endxy $$
All of the other natural transformations are defined similarly. It is easy to see that $\L^{\A} = (L^{\A}, \vec{\e}^{\A}, \vec{\d}^{\A})$ and $\R^{\A} = (R^{\A}, \vec{\eta}^{\A}, \vec{\mu}^{\A})$ define an awfs because all the definitions are given by simply post-composing a natural transformation with the old comonad and monad.

\begin{rmk} Note however that the underlying wfs of the pointwise awfs $(\L^{\A}, \R^{\A})$ is not itself given pointwise by the underlying wfs of $(\L,\R)$. This is because, unlike the case for the left and right factors, generic pointwise maps will not have natural lifts. This is one area where awfs behave better than ordinary wfs.
\end{rmk} 

\subsection{Cofibrantly generated case} 

Given a cofibrantly generated awfs $(\L,\R)$ on $\M$, is the resulting pointwise awfs $(\L^{\A},\R^{\A})$ on $\M^{\A}$ cofibrantly generated? There are many reasons to suspect that this is not the case. For example, there is an awfs on $\Set$ generated by $\J=\{\emptyset \ra 1\}$ for which the right class is the epimorphisms. The right class of the pointwise awfs on $\Set^{\A}$ consists of epis with a natural section. If this awfs is cofibrantly generated, it means that this class can be characterized by a lifting property. While right lifting properties can be used to specify additional \emph{structure} on a class of maps, they are not typically known to impose \emph{coherence} conditions.

Despite this worry, the answer is yes, the pointwise awfs is always cofibrantly generated when the original one is. In retrospect, the solution to the above concern is obvious: the generating category $\J_{\A}$ for the pointwise awfs will not be discrete (unless $\A$ is)!
This is the first example known to the author where the extra generality allowed in Garner's small object argument is useful.

\begin{thm}\label{cgiawfs} Let $J \colon \J \ra \M^{\bf 2}$ be a small category over $\M^{\bf 2}$, where $\M$ permits the small object argument, and let $(\L,\R)$ be the awfs generated by $\J$. Let $\J_{\A}$ be the category $\A^{\op} \times \J$ equipped with the functor $$\A^{\op} \times \J \sr{y \times J}{\lra} \Set^{\A} \times \M^{\bf 2} \sr{-\cdot-}{\lra} (\M^{\A})^{\bf 2},$$ where $y$ denotes the Yoneda embedding and $-\cdot-$ denotes the copower\footnote{The copower $S\cdot j$ of a set $S$ with an object $j$ of $\M^{\bf 2}$ is the coproduct  of copies of $j$ indexed by the set $S$. When $S$ is instead a set-valued functor, the copower $S \cdot j$ is a natural transformation with each constituent arrow having the description just given.} (tensor) of an arrow in $\M$ by a $\Set$-valued functor from $\A$. Then the pointwise awfs $(\L^{\A},\R^{\A})$ is generated by $\J_{\A}$.
\end{thm}

In keeping with the previous notational conventions, we regard $\J_{\A}$ as the category with objects $\A(a,-) \cdot j$ for $a \in \A$ and $j \in \J$. Morphisms are generated by maps $\A(a,-) \cdot j \Ra \A(a, -) \cdot j'$ for every $j \Ra j'$ in $\J$ and by maps $f^* \colon \A(b,-) \cdot j \Ra \A(a,-) \cdot j$ for every $f \colon a \ra b$ in $\A$.  We prefer to write $J_{\A} \colon \J_{\A} \ra (\M^{\A})^{\bf 2}$ for the composite functor defined above.

\begin{proof}[Proof of Theorem \ref{cgiawfs}] We don't know a priori whether $\M^{\A}$ permits the small object argument, but we can begin to apply that construction to the category $\J_{\A}$ over $(\M^{\A})^{\bf 2}$ nonetheless. We will show that the functors $(L^{\A})^0$, $(L^{\A})^1$, $(R^{\A})^1$, $(L^{\A})^2$, $(R^{\A})^2$, etc that arise at each step agree with the functors $L^0$, $L^1$, $R^1$, etc pointwise. It will follow that our construction on $(\M^{\A})^{\bf 2}$ converges to the awfs $(\L^{\A}, \R^{\A})$, which is therefore generated by $\J_{\A}$.

The beginning stage of the small object argument computes the step-zero comonad $(L^{\A})^0$ as the left Kan extension of $J_{\A} \colon \J_{\A} \ra (\M^{\A})^{\bf 2}$ along itself. Note that $(\M^{\A})^{\bf 2}$ is cocomplete, since $\M$ is. The familiar formula for Kan extensions gives
\begin{align*} (L^{\A})^0 \a &= \int^{(a, j) \in \J_{\A}\cong \A^{\op} \times \J} \Hom_{(\M^{\A})^{\bf 2}}( \A(a,-)\cdot j, \a) \cdot (\A(a,-) \cdot j). \\ \intertext{The natural transformation $\A(a,-)\cdot j$ is the image of $j$ under a functor $\M^{\bf 2} \ra (\M^{\A})^{\bf 2}$ that is left adjoint to evaluation at $a$. By this adjunction, the above coend equals} &= \int^{\J_{\A}\cong \A^{\op} \times \J} \Hom_{\M^{\bf 2}}(j, \a_a) \cdot (\A(a,-)\cdot j) \\ &= \int^{\J_{\A}\cong \A^{\op}\times \J} \Sq(j,\a_a)\cdot (\A(a,-)\cdot j) \\ \intertext{where we've written ``Sq'' to indicate that morphisms from $j$ to $\a_a$ in $\M^{\bf 2}$ are commutative squares. By Fubini's theorem and cocontinuity of the copower, we can use the isomorphism $\J_{\A} \cong \A^{\op} \times \J$ to compute the coend over $\J$ first, yielding} &= \int^{\A^{\op}} \A(a,-) \cdot \left( \int^{\J} \Sq(j,\a_a) \cdot j \right) \\ &= \int^{\A^{\op}} \A(a,-) \cdot L^0 \a_a \\ \intertext{where $L^0$ is the step-zero comonad for $\J$. We now express this coend as a coequalizer} &= \text{coeq} \left( \coprod_{f \colon a \ra b} \A(b,-) \cdot L^0 \a_a \rightrightarrows \coprod_a \A(a,-) \cdot L^0 \a_a \right)
\end{align*}
where the top arrow is induced by $f^* \colon \A(b,-) \ra \A(a,-)$ and the bottom arrow is induced by $L^0$ applied to the naturality square for $f$, which is a morphism from $\a_a$ to $\a_b$ in $\M^{\bf 2}$. We compute this coequalizer pointwise; by inspection at an object $c \in \A$, the coequalizer in $\M^{\bf 2}$ is $L^0\a_c$ with $\A(a,c) \cdot L^0\a_a \Ra L^0\a_c$ given by the evaluation map. This object and morphism satisfy the required universal property: the map out of $L^0 \a_c$ can be found by restricting to the identity component of the copower $\A(c,-)\cdot L^0 \a_c$. 

The remaining steps in the small object argument are constructed from previous ones by applying the comonad $(L^{\A})^0$ and taking pushouts, coequalizers, and transfinite composites, which are all computed pointwise. As we've shown that $(L^{\A})^0$ is also computed pointwise, we are done. Because $\M$ permits the small object argument, this process will converge for each arrow $\a_a$ at some time (ordinal) $\b$, which means that the naturally constructed arrows from step $\b$ to step $\b+1$ are isomorphisms. It follows that there is a natural isomorphism from step $\b$ on $(\M^{\A})^{\bf 2}$ to step $\b+1$, which tells us that the construction converges. This completes the proof that Garner's small object argument applied to $\J_{\A}$ will give the pointwise monad and comonad of $(\L^{\A}, \R^{\A})$. Hence, $(\L^{\A}, \R^{\A})$ is the awfs generated by $\J_{\A}$.
\end{proof}

\begin{ex}\label{trivmodelstrex} When $\A$ is a small category, the awfs of Lack's trivial model structure on the 2-category $\Cat^{\A}$ are pointwise awfs. In the case $\A = {\bf 2}$, Lack proves \cite[Proposition 3.19]{lackhomotopy} that his trivial model structure is not cofibrantly generated in Quillen's sense. By contrast, Theorem \ref{cgiawfs} can be used to show that this is an algebraic model structure, which is cofibrantly generated in Garner's sense.
\end{ex}

\subsection{Algebraic projective model structures} 

Using Theorem \ref{adjthm} and the pointwise algebraic weak factorization system described above, we can prove that any cofibrantly generated algebraic model structure on a category $\M$ induces a cofibrantly generated \emph{projective algebraic model structure} on the diagram category $\M^{\bf \A}$. The awfs of this model structure are not the pointwise awfs on $\M^{\A}$; instead, the generating categories are discrete, at least when the original generators $\I$ and $\J$ are. The underlying model structure agrees with the usual projective model structure on a diagram category: weak equivalences are pointwise weak equivalences and fibrations are pointwise fibrations.

The generating categories $\I_{\mathrm{proj}}$ and $\J_{\mathrm{proj}}$ for the projective model structure look familiar; in the case where $\I$ and $\J$ are discrete these are the usual generating sets in the classical theory. Objects of $\I_{\mathrm{proj}}$ are functors $\A(a,-) \cdot i$, for all $a \in \A$ and $i \in \I$. Each morphism $i \Ra i'$ in $\I$ gives rise to a morphism $\A(a,-) \cdot i \Ra \A(a,-) \cdot i'$ in $\I_{\mathrm{proj}}$; there are no others. The category $\J_{\mathrm{proj}}$ is defined similarly.

\begin{thm} Let $\M$ have an algebraic model structure, generated by $\I$ and $\J$, with weak equivalences $\W_{\M}$. Then the categories $\I_{\mathrm{proj}}$ and $\J_{\mathrm{proj}}$ give rise to a cofibrantly generated algebraic model structure on $\M^{\A}$, which we will call the projective algebraic model structure.
\end{thm}
\begin{proof}
Write $\A_0$ for the discrete subcategory of objects of $\A$. We first show that the algebraic model structure on $\M$ induces a model structure on the diagram category $\M^{\A_0}$. We then use an adjunction to pass this across to the projective model structure on $\M^{\A}$. 

Arrows of $\M^{\A_0}$ are natural transformations with no naturality conditions, i.e., collections $\a$ of morphisms $\a_a$ in $\M$ for each $a \in \A_0$. The categories $\I$ and $\J$ induce a pair of pointwise awfs on $\M^{\A_0}$. By Theorem \ref{cgiawfs}, these awfs are generated by $\I_{\A_0}$ and $\J_{\A_0}$.\footnote{Of course, it is also possible to prove this directly as an easier special case of that theorem.} The comparison map of the algebraic model structure on $\M$ gives the elements of $\J$ coalgebra structures for the comonad generated by $\I$ that are natural with respect to morphisms in $\J$. Pointwise, this functor can be used to define a functor from $\J_{\A_0}$ to the category of coalgebras for the comonad induced by $\I_{\A_0}$. By Remark \ref{comprmk}, this induces a comparison map between the awfs generated by $\J_{\A_0}$ and $\I_{\A_0}$.

We quickly prove that this gives an algebraic model structure. As in the proof of Theorem \ref{adjthm}, the existence of this comparison map implies that trivial cofibrations are cofibrations and trivial fibrations are fibrations.  Let $\W_0$ be the class of morphisms of $\M^{\A_0}$ that are pointwise weak equivalences. With this definition it is clear that trivial cofibrations and trivial fibrations are weak equivalences. So to show that $\I_{\A_0}$ and $\J_{\A_0}$ give rise to an algebraic model structure, it remains only to show that fibrations that are weak equivalences are trivial fibrations.

More precisely, we need to show is that algebras for the monad induced by $\J_{\A_0}$ that are pointwise weak equivalences have an algebra structure for the monad induced by $\I_{\A_0}$. Since the category $\A_0$ is discrete, a collection $\a$ of morphisms $\a_a$ has an algebra structure for the monad induced by $\J_{\A_0}$ just when each $\a_a$ is an algebra for the monad induced by $\J$. Here, each $\a_a$ is a trivial fibration for the algebraic model structure on $\M$; by Lemma \ref{retractlem} this means that it has an algebra structure for the monad induced by $\I$. Again because $\A_0$ is discrete, this means that the collection $\a$ has an algebra structure for the monad induced by $\I_{\A_0}$, which is what we wanted to show. So the categories $\I_{\A_0}$ and $\J_{\A_0}$ generate an algebraic model structure on $\M^{\A_0}$.

Let $i \colon \A_0 \hookrightarrow \A$ be the canonical inclusion. Then left Kan extension along $i$ gives rise to an adjunction $$\xymatrix{ \mathrm{Lan}_i \colon \M^{\A_0} \ar@<1ex>[r] \ar@{}[r]|-{\perp} & \M^{\A} \colon i^* \ar@<1ex>[l]}$$ Here $i^*$ might be thought of as an ``evaluation'' map; it takes a functor $G\colon \A \ra \M$ to the collection of objects in its image and a natural transformation $\a$ to its collection of constituent arrows. Using the usual formula for left Kan extensions, the left adjoint takes an arrow $\a \in \M^{\A_0}$ to the disjoint union $\displaystyle\sqcup_{c \in \A_0} \A(c,-) \cdot \a_c$. Objects in $\I_{\A_0}$ are natural transformations $\A_0(a,-) \cdot i$ for some $a \in \A$ and $i \in \I$. As $\A_0$ is discrete, this natural transformation consists of the arrow $i$ at the component for $a$ and the identity arrow at the initial object of $\M$ at all other objects of $\A$. The image of this object under Lan$_i$ is $\A(a,-) \cdot i$, by the above formula.  It follows that $$\mathrm{Lan}_i\, \I_{\A_0} = \I_{\mathrm{proj}} \hspace{.5cm} \mathrm{and} \hspace{.5cm} \mathrm{Lan}_i\, \J_{\A_0} = \J_{\mathrm{proj}}.$$

In order to apply Theorem \ref{adjthm} and conclude that $\M^{\A}$ has an algebraic model structure generated by $\I_{\mathrm{proj}}$ and $\J_{\mathrm{proj}}$, we must show that the right adjoint $i^*$ takes the underlying maps of the coalgebras for the comonad generated by $\J_{\mathrm{proj}}$ to weak equivalences in $\M^{\A_0}$. In other words, we must show that the coalgebras for the comonad generated by $\J_{\mathrm{proj}}$ are pointwise weak equivalences. 

Coalgebras for the comonad generated by $\J_{\mathrm{proj}}$ are in the left class of the underlying wfs   that this category generates, that is, they are arrows satisfying the LLP with respect to the underlying class of $\J_{\mathrm{proj}}^{\boxslash}$.  From the adjunction, we know that the underlying class of $\J_{\mathrm{proj}}^{\boxslash} = (\mathrm{Lan}_i\, \J_{\A_0})^{\boxslash} = (i^*)^{-1}(\J_{\A_0}^{\boxslash})$ is the class of pointwise algebras for the original monad generated by $\J$ on $\M^{\bf 2}$. 

Let $\a$ be an element of the left class generated by $\J_{\mathrm{proj}}$ and factor $\a$ using the pointwise awfs $(\CC_t^{\A}, \FF^{\A})$ on $\M^{\A}$, not the awfs generated by $\J_{\mathrm{proj}}$. The components of the right factor $F^{\A}\a$ are algebras for the monad $\FF$ generated by $\J$ because $F^{\A}\a$ is an algebra for the monad $\FF^{\A}$. So $F^{\A}\a \in \J_{\mathrm{proj}}^{\boxslash}$ and hence $\a$ lifts against $F^{\A}\a$, which means that $\a$ is a retract of $C_t^{\A}\a$. The constituent maps $(C_t^{\A}\a)_a = C_t (\a_a)$ are coalgebras for the comonad on $\M^{\bf 2}$ generated by $\J$; in particular they are weak equivalences, since $\J$ is the generating category of trivial cofibrations. So pointwise the arrows of $\a$ are retracts of weak equivalences; hence $\a$ consists of pointwise weak equivalences. Theorem \ref{adjthm} may now be used to establish the projective algebraic model structure.
\end{proof}

\section{Recognizing cofibrations}\label{appsec}

In previous sections, we have seen that cofibrantly generated model structures can be ``algebraicized,'' so that the constituent wfs are in fact awfs. This gives all fibrations the structure of algebras for a monad and some cofibrations the structure of coalgebras for a comonad. This extra algebraic structure is unobtrusive, in the sense that it can be forgotten at any point to yield an ordinary notion of a model structure, with the added benefit that the factorizations constructed by Garner's small object argument are somehow ``smaller.''

However, we have not yet given a convincing argument that this extra algebra structure is useful, allowing us to prove theorems that were intractable otherwise. In this section, we will provide the first such examples, illustrating the following point: one pleasant feature of this algebraic data is it gives a technique for proving that certain maps are cofibrations.

\subsection{Coalgebra structures for the comparison map} 

One such example is the following theorem, which is joint work with Richard Garner and Mike Shulman. In this theorem, we will require that the comparison map arise from a functor $\t\colon \J \ra \I$ over $\M^{\bf 2}$ of a particularly nice form. Firstly, we require that it be a full inclusion (full, faithful, and injective on objects). Secondly, we require that $\I$ decompose as a coproduct $\t(\J) \sqcup \I'$, i.e., that there are no morphisms from objects in the image of $\t$ to objects not in the image.\footnote{We suspect that this second condition is unnecessary but include it to simplify the arguments given below.} Note that when $\J$ and $\I$ are sets, these requirements simply mean that the generating trivial cofibrations $\J$ are a subset of the generating cofibrations $\I$.

\begin{thm}\label{cofthm} Let $\J$ and $\I$ be categories that generate an algebraic model structure on $\M$ and such that we have an inclusion $\t\colon \J \ra \I$ over $\M^{\bf 2}$ of the form described above. Suppose also that the cofibrations are monomorphisms in $\M$. Then the components of the comparison map $\xi \colon (\CC_t, \FF) \rightarrow (\CC, \FF_t)$ produced by Garner's small object argument are cofibrations and, furthermore, coalgebras for the comonad $\CC$.
\end{thm}

Let us provide some intuition for this result. Given an arrow $f$, we construct $Qf$ from $Rf$ by attaching more ``cells.'' Because the cofibrations are monomorphisms, the ``cells'' we had attached previously to form $Rf$ are not killed by the quotienting involved in the construction of $Qf$. Hence the arrow $\xi_f \colon Rf \ra Qf$ is itself a cofibration, and furthermore, because it was constructed cellularly, $\xi_f$ is a $\CC$-coalgebra.

It takes some effort to describe the comparison map explicitly and accordingly it will take some work to translate the above intuition into a rigorous argument.  When the awfs are cofibrantly generated, the comparison map $\xi \colon (\CC_t, \FF) \ra (\CC, \FF_t)$ is induced by the cone produced by the right-hand factorization over the colimits of the left-hand factorization. These are each constructed by various colimiting processes at a number of stages, and the proof will accordingly involve a transfinite induction corresponding to each stage. 

Specifically, for each ordinal $\a$, the small object argument produces functorial factorizations $(C^{\a}_t, F^{\a})$ and $(C^{\a}, F^{\a}_t)$. Let $Q^{\a},R^{\a} \colon \M^{\bf 2} \ra \M$ denote the functors accompanying each functorial factorization, i.e., so that $f \colon X \ra Z$ factors as $$\xymatrix{ & X \ar[dl]_-{C^{\a}_tf} \ar[dr]^-{C^{\a}f} & \\ R^{\a} f \ar[rr]^-{\xi^{\a}_f} \ar[dr]_-{F^{\a}f} & & Q^{\a}f \ar[dl]^-{F^{\a}_tf} \\ & Z}$$ with $\xi^{\a}$, the component at $f$ of the step-$\a$ comparison map, as depicted.

Recall, the step-one factorization $(C_t^1,F^1)$ is constructed by factoring the counit of the density comonad of $J \colon \J \ra \M^{\bf 2}$ as a pushout followed by a square with domain equal to the identity, as indicated below.
\begin{equation}\label{step0and1}\xymatrix{ \cdot \ar[r] \ar[d]_{C_t^0f} \ar@{}[dr]|(.8){\ulcorner}& X \ar[d]^{C_t^1 f} \ar@{=}[r] & X \ar[d]^f \\ \cdot \ar[r] & R^1f \ar[r]_{F^1f} & Z}\end{equation}
In the familiar case when $\J$ is discrete, $C_t^0f=$Lan$_JJ(f)$ is the coproduct of elements $j \in \J$ over commutative squares from $j$ to $f$, and the top and bottom horizontal composites are the canonical arrows induced from these coproducts. 

A key step in the proof of Theorem \ref{cofthm} is the following lemma, which will imply that the step-one comparison map $\xi^1_f \colon R^1 f \ra Q^1 f$ is a $\CC$-coalgebra. The general form of this lemma will enable multiple applications.

\begin{lem}\label{coflem} Let $\J$ and $\I$ be small categories with an inclusion $\J \ra \I$ over $\M^{\bf 2}$ as described above and let  $(C^1_t, F^1)$ and $(C^1, F^1_t)$ be the step-one factorizations they produce. Given any commutative triangle $\raisebox{.25in}{\xymatrix{ X \ar@{ >->}[rr]^h \ar[dr]_f && Y \ar[dl]^g \\ & Z &}}$ such that $h$ is a cofibration and a monomorphism, then the map $\xi^1 \colon R^1 f \ra Q^1 g$ induced by the colimit is a cofibration. If furthermore $h$ is a $\CC$-coalgebra, then so is $\xi$.
\end{lem} 

\begin{rmk}\label{coalgfactsrmk}
The proof of Lemma \ref{coflem} will require some basic facts about coalgebras for a comonad. We say a morphism $(u,v) \colon f \Ra g$ in $\M^{\bf 2}$ is a map of $\CC$-coalgebras if it lifts to the category $\Calg$ (where we usually have particular coalgebra structures for $f$ and $g$ in mind).  In particular, if $f$ has a coalgebra structure and $g$ is a pushout of $f$, the pushout square is a map of coalgebras, when $g$ is given the canonical coalgebra structure of the pushout. (See the example in Footnote \ref{pushfootnote}.) Similarly, if $g$ is a colimit of any diagram in $\M^{\bf 2}$ whose objects are coalgebras and whose arrows are maps of coalgebras, then $g$ inherits a canonical coalgebra structure such that the legs of the colimit cone are maps of coalgebras. (This is a consequence of Theorem \ref{closureprop}.)  Finally, when $\CC$ is the comonad of an awfs, $\CC$-coalgebras are closed under composition, as we saw in Lemma \ref{complem}, in such a way that $(1,g) \colon f \Ra gf$ is a map of coalgebras if $f$ and $g$ are coalgebras.

We will use all of these facts in the proofs that follow.
\end{rmk}

\begin{proof}[Proof of Lemma \ref{coflem}] The defining pushouts for the arrows $C_t^1f$ and $C^1g$ are the top and bottom faces of the cube below.
\begin{equation}\label{lemmacube}{\small \xymatrix{ & \cdot \ar'[d][dd]^{i_S} \ar[rr]^{e} \ar@{ >->}[dl]_{C_t^0f} \ar@{}[dr]|(.8){\ulcorner} & & X \ar@{ >->}[dl]_{C_t^1 f} \ar@{ >->}'[dd]+U*{}^{h} & \\ \cdot \ar[rr]^(.7){p} \ar[dd]_{i_D} & & R^1 f \ar@{-->}[dd]^(.3){\xi^1} & & \\ & \cdot \ar@{ >->}[dl]_(.3){C^0g} \ar'[r]_{e'}[rr] \ar@{}[dr]|(.8){\ulcorner}&& Y \ar@{ >->}[dl]^{C^1g} & \\ W \ar[rr]_{p'} & & Q^1 g  & & }}\end{equation} The notation $(i_S, i_D)$ is meant to evoke the inclusions of the indexing sets for the coproducts of spheres and disks, with the familiar generating set of cofibrations $\{S^{n-1} \ra D^n\}$ in mind. The map $\xi^1$ is induced by the universal property of the top pushout.

We begin by defining the pushout in the right face of the cube (\ref{lemmacube}).
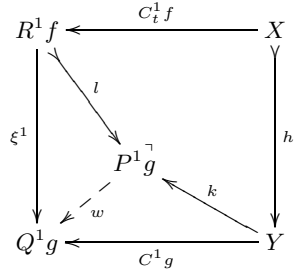
$${\small \xymatrix@=15pt{ R^1f \ar[ddd]_{\xi^1} \ar@{ >->}[ddr]^l & & & X \ar[lll]_{C^1_tf} \ar@{ >->}[ddd]^h \ar@{}[ddll]|(.9){\urcorner} \\ \\ & P^1g \ar@{-->}[dl]^w \\ Q^1g & & & Y \ar[ull]_k \ar[lll]^{C^1g} }}$$
Because $h$ is a cofibration, $l$ is as well. If $h$ is a coalgebra, then $l$ inherits a canonical coalgebra structure as a pushout of $h$. Because cofibrations and coalgebras for the comonad of an awfs are closed under composition, it suffices to show that $w$ is a cofibration, and a coalgebra whenever $h$ is a coalgebra. Actually $w$ is always a coalgebra. We will use:

\begin{lem}\label{pushoutlemma} Given a commutative cube in which the top and bottom faces are pushouts, form the pushouts in the left and right faces, as in the diagram below.
$${\small 
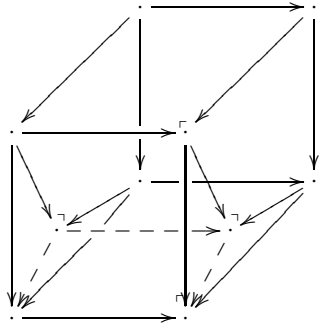
\xymatrix@=.75pc{ & && \cdot \ar[rrrr] \ar[dddlll] \ar'[ddd][dddd] \ar@{}[dddddll]|(.95){\urcorner} \ar@{}[dddr]|(.95){\ulcorner}& &&  & \cdot \ar[dddlll] \ar[dddd] \ar@{}[dddddll]|(.95){\urcorner} \\ & & && && & \\ &&  && && &  \\ \cdot \ar[dddd] \ar[rrrr] \ar[ddr] & && & \cdot \ar[dddd] \ar[ddr] &  & & \\ &  & & \cdot \ar[dll] \ar'[r][rrrr] \ar'[dl][dddlll] \ar@{}[dddr]|(.87){\ulcorner}& & & & \cdot \ar[dll] \ar[dddlll] \\   & \cdot \ar@{-->}[ddl] \ar@{-->}'[rrr][rrrr] & &&  & \cdot \ar@{-->}[ddl] & &  \\& & &&&&& \\ \cdot \ar[rrrr] & & & & \cdot && &}}$$
Then the square created by these pushouts (with three edges dotted in the diagram) is itself a pushout square.
\end{lem}
\begin{proof} Easy diagram chase.
\end{proof}

By Lemma \ref{pushoutlemma}, $w$ is a pushout of the map $c$ in the diagram below,
$${\small 
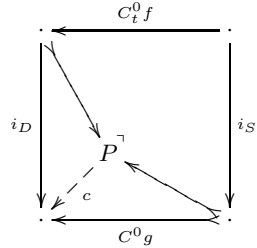
\xymatrix@=15pt{ \cdot \ar[ddd]_{i_D} \ar@{ >->}[ddr] & & & \cdot \ar[lll]_{C^0_tf} \ar[ddd]^{i_S} \ar@{}[ddll]|(.9){\urcorner} \\ \\ & P \ar@{-->}[dl]^c \\ \cdot & & & \cdot \ar@{ >->}[ull] \ar@{ >->}[lll]^{C^0g} }}$$
so it remains to show that this is a cofibration and coalgebra. 

In the case where $\J$ and $\I$ are discrete, $C^0_tf$ is the disjoint union of arrows $j$ of $\J$ over commutative squares from $j$ to $f$, and $C^0g$ is the disjoint union of arrows $i$ of $\I$ over squares from $i$ to $g$. The square $(i_S, i_D)$ maps the first coproduct into the second, sending $j$ to its image under the functor $\t \colon \J \ra \I$ and a square from $j$ to $f$ to the composite of this square with $(h,1) \colon f \Ra g$. As $h$ is monic, $(i_S, i_D)$ is an inclusion, so we can separate the coproduct $C^0g$ into the image of this inclusion and the rest. If we write $h_* \colon \Sq(\J,f) \ra \Sq(\I,g)$ for the (injective) function that takes a square from $j$ to $f$ for some $j \in \J$ and composes with $(h,1)$ to get a square from $\t j$ to $g$, then
$$C^0 g = \left( \bigsqcup_{\Sq(\J,f)} \t j \right) \sqcup \left( \bigsqcup_{\Sq(\I,g) \backslash \text{im}\, h_*} i \right)$$
With this notation, $C^0g$ factors through $P$ as the composite of first the arrow $$\left( \bigsqcup_{\Sq(\J,f)} \t j \right) \sqcup \left( {\bigsqcup_{\Sq(\I,g) \backslash \text{im}\, h_*}} 1_{\dom i} \right)  \text{then} \left( \bigsqcup_{\Sq(\J,f)} 1_{\cod \t j} \right) \sqcup \left( \bigsqcup_{\Sq(\I,g) \backslash \text{im}\, h_*} i \right).$$
This second arrow is $c$, which gets a canonical $\CC$-coalgebra structure as a coproduct of arrows in $\I$ with identities, which are always coalgebras.

The proof of the general case where $\J$ and $\I$ are not discrete is similar. In this case, $C^0_tf$ is a quotient of the disjoint union described above, and similarly for $C^0g$. But $C^0g$ can still be separated into the disjoint union of the image of $(i_S,i_D)$ and its complement by the hypotheses we made on the inclusion $\J \ra \I$. So $c$ has essentially the same description as above, except it is a quotient of a coproduct. This time, $c$ is a colimit of a diagram whose objects are either identities or generating cofibrations, so to prove $c$ is a coalgebra we must check that the maps of the diagram are maps of coalgebras. This is true because the morphisms in the formula for a coend are either arrows of $\I$, which are canonically coalgebra maps, or they are coproduct inclusions, which are always coalgebra maps. In any case, the above conclusion still stands: $c$ is canonically a coalgebra, so $w$ is as well. Hence $\xi^1$ is a cofibration, and a $\CC$-coalgebra when $h$ is.
\end{proof}

\begin{rmk} It is possible to prove directly that $\xi^1$ is a cofibration by showing that it lifts against all trivial fibrations. But this proof can only show $\xi^1$ is a cofibration, not that it has a $\CC$-coalgebra structure when $h$ does, and it is this stronger fact that we will need in the proof of Theorem \ref{cofthm}.
\end{rmk}

\begin{rmk}\label{caveatrmk} The argument of Lemma \ref{coflem} holds more generally than stated. In particular, it is not necessary that the arrows in the positions of $C^0_tf$ and $C^0g$ be coends over \emph{all} possible squares. As long as these arrows are constructed as coends over \emph{some} squares such that $(i_S,i_D)$ is an inclusion, the conclusion follows. In applications, we will often require this slightly more general result, for reasons that will become clear in a moment.
\end{rmk}

We will now use Lemma \ref{coflem} to prove Theorem \ref{cofthm}. Our proof used a modified version of the small object argument, suggested by Richard Garner in private communication, that can be used whenever the elements of the left class of the underlying wfs are monomorphisms. Steps zero and one, as depicted in (\ref{step0and1}) are the same as before. At this point, Quillen's small object argument has us freely attach ``cells'' to fill ``spheres'' in the object $R^1f$ by repeating steps zero and one for the map $F^1f$. Garner's small object argument does the same thing, but then takes a coequalizer to identify the ``cells'' in the ``spheres'' that were filled twice, once in step one and once in step two. In the modified version, we never attach these extraneous ``cells'' at all; instead,  we only attach ``cells'' to fill ``spheres'' in $R^1f$ that weren't filled already in step one. For this modification to work, it is essential that the cofibrations are monomorphisms; otherwise, ``cells'' that are needed to fill ``spheres'' at some intermediate stage may become redundant later. With this modification, solutions to lifting problems $\J \boxslash Ff$ factor uniquely through a minimal stage; colloquially every ``sphere'' that is filled in the object $Rf$ is filled at some minimal step and is filled uniquely at this step. This gives a new form of the small object argument that produces the same factorizations as in Garner's version but with no need for taking coequalizers, which are the most difficult to manipulate. This modification of the small object argument was independently suggested by  \cite{radulescubanu} in the special case of fibrant replacement.

\begin{proof}[Proof of Theorem \ref{cofthm}] We use the preceding lemma and transfinite induction. Applying Lemma \ref{coflem} in the case $h= 1_X$ and $f=g$ shows that $\xi^1_f \colon R^1 f \ra Q^1f$ is a cofibration and a $\CC$-coalgebra. Because the cofibrations are assumed to be monomorphisms, we may use the modified version of Garner's small object argument described above. In the modified version, step two applies factorizations that are similar to the step one factorizations to each of the vertical morphisms in the triangle \begin{equation}\label{iterativetriangle}\xymatrix{ R^{\b}f \ar@{ >->}[rr]^{\xi^{\b}_f} \ar[dr]_{F^{\b}f} & & Q^{\b} f \ar[dl]^{F_t^{\b}f} \\ & Y &}\end{equation} with ordinal $\b=1$ in this case.  The difference between these factorizations and the step one factorizations is that some squares are left out of the step zero coproducts. By Remark \ref{caveatrmk}, we nonetheless deduce that $\xi^2_f$ is a $\CC$-coalgebra. Likewise, applying Lemma \ref{coflem} to the triangles (\ref{iterativetriangle}) produced at each stage, we conclude that each map $\xi_f^{\b+1} \colon R^{\b+1} f \ra Q^{\b+1} f$ is a $\CC$-coalgebra, assuming $\xi_f^{\b}$ is.

For limit ordinals $\a$, the maps $R^{\a} f \ra Q^{\a} f$ are created as colimits of the diagrams of the $\xi^{\b}_f\colon R^{\b}f \ra Q^{\b}f $ for ordinals $\b < \a$, which by the inductive hypothesis are cofibrations and $\CC$-coalgebras.
As usual, we must check that the morphisms $\xi_f^{\b} \Ra \xi_f^{\b+1}$ in this diagram are maps of coalgebras. When we apply Lemma \ref{coflem} to (\ref{iterativetriangle}), the square $\xi_f^{\b} \Ra \xi_f^{\b+1}$ is the right face of the cube (\ref{lemmacube}), reproduced below 
$$\xymatrix@=15pt{ \cdot \ar[ddd]_{\xi^{\b+1}} \ar@{ >->}[ddr]^l & & & \cdot \ar[lll]_{c} \ar@{ >->}[ddd]^{\xi^{\b}} \ar@{}[ddll]|(.9){\urcorner} \\ \\ & \cdot \ar@{-->}[dl]^w \\ \cdot & & & \cdot \ar[ull]_k \ar[lll]^{d} }$$
The arrow $l$ inherits its coalgebra structure as a pushout of $\xi^{\b}$, and this construction makes $(c,k)$ a map of coalgebras. Similarly, $\xi^{\b+1}$ inherits its coalgebra structure as a composite of the coalgebras $l$ and $w$, and this construction makes $(1,w)$ a map of coalgebras. The morphism $(c,d) \colon \xi^{\b}_f \Ra \xi^{\b+1}_f$ is a composite of maps of coalgebras, and hence a map of coalgebras. Hence the colimit $\xi^{\a}_f$ has a canonical coalgebra structure created by this diagram. By transfinite induction, the comparison map $\xi$ is a pointwise cofibration with each component a $\CC$-coalgebra.
\end{proof}

\subsection{Algebraically fibrant objects revisited}  

One consequence of Lemma \ref{coflem} and the proof of Theorem \ref{cofthm} is the following corollary, which says that fibrant replacement monads that are constructed algebraically preserve certain trivial cofibrations, assuming the trivial cofibrations are monomorphisms. 

\begin{cor}\label{fibrantcor} Let $\M$ be a category that permits the small object argument equipped with a model structure such that the trivial cofibrations are monomorphisms. Suppose there exists a category $\J$ of trivial cofibrations that detects algebraically fibrant objects, in the sense that an object $X$ is fibrant if and only if $X \ra *$ underlies some object of $\J^{\boxslash}$, and let $\R$ be the fibrant replacement monad on $\M$ induced from the awfs $(\CC_t,\FF)$ generated by $\J$. Then if $f \colon X \ra Z$ is a $\CC_t$-coalgebra, $Rf \colon RX \ra RZ$ is a $\CC_t$-coalgebra.
\end{cor}
\begin{proof}
The arrow $Rf$ is constructed by an inductive process, analogous to the construction of the comparison map, that begins by applying Lemma \ref{coflem} to the triangle $$\xymatrix{ X \ar@{ >->}[rr]^f \ar[dr] & & Z \ar[dl] \\ & {*} & }$$ with both awfs taken to be $(\CC_t,\FF)$.
Hence this result. 
\end{proof}

The algebras for the monad $\R$ are precisely the ``algebraically fibrant objects'' of $\M$, i.e., objects with chosen lifts against the generators, subject to any coherence conditions imposed by morphisms in the category $\J$. By Lemma \ref{retractlem} every fibrant object has some algebra structure making it an algebraically fibrant object. It always suffices to take $\J$ to be the generating trivial cofibrations, assuming they exist, but in some examples it is preferable to use a smaller generating category. 

As for any category of algebras for a monad, we have an adjunction $$\xymatrix{\M \ar@<1ex>[r]^-T \ar@{}[r]|-{\perp} & \Ralg \ar@<1ex>[l]^-U}$$ where $T$ takes an object $X \in \M$ to the free $\R$-algebra $(RX, \mu_X)$. In practice, the category $\Ralg$ may fail to be cocomplete, in which case it is not a suitable category for a model structure. But when $\Ralg$ is cocomplete, as is the case when $\M$ is locally presentable and $\R$ arises from a cofibrantly generated awfs for example, Corollary \ref{fibrantcor} provides some hope that one could build a model structure on $\Ralg$ such that $T \dashv U$ is a Quillen adjunction. One feature of such a model structure is that its objects would all be fibrant. In fact, it follows easily that any such Quillen adjunction is in fact a Quillen equivalence.  

One such example is the Quillen model structure on  $\Ch_{A}$, the category of chain complexes of $A$-modules for some commutative ring $A$, though this example is rather unsatisfactory because the objects of $\Ch_{A}$ are already fibrant. More interestingly, Thomas Nikolaus has proven that the categories of algebraic Kan complexes and algebraic quasi-categories can be given a model structure, lifted in the first case from Quillen's and the second from Joyal's model structure on simplicial sets \cite{nikolausalgebraic}. For the latter case, we prefer to let the set $\J$ in Corollary \ref{fibrantcor} be the inner horn inclusions, rather than the generating trivial cofibrations. For all of these examples, Theorems \ref{adjthm} and \ref{adjthm5} imply that the resulting model structures are algebraic and the Quillen equivalences are algebraic Quillen adjunctions.

\section{Adjunctions of awfs}\label{adjsec}

We now return to the material previewed in Section \ref{algquillenssec}. In this section, we study adjunctions of awfs, which we define precisely below. To motivate this definition, we consider an important class of examples: suppose $\J$ generates an awfs $(\CC,\FF)$ on $\M$ and $T\J$ generates an awfs $(\L,\R)$ on $\K$, where $T\colon \M \ra \K$ is the left adjoint of a specified adjunction. A main task of this section is to prove Theorem \ref{adjawfsthm}, which says that there is a canonical adjunction of awfs in this situation.

A direct proof is possible but technically difficult. Instead, we present a more conceptual, though somewhat circuitous argument, that is nonetheless shorter. After preliminary explorations, we reintroduce the three notions of morphisms between awfs on different categories, each extending Definition \ref{morawfsdefn}. In order to prove Theorem \ref{adjawfsthm}, we use Theorem \ref{compcharthm}, which says that an awfs $(\L,\R)$ is equivalently characterized by a natural composition law on the category of algebras for a monad over cod. We prove a lemma that allows us to use this recognition principle to easily identify lax morphisms of awfs, which for categorical reasons, suffices to prove Theorem \ref{adjawfsthm}.

However, Theorem \ref{adjawfsthm} is not quite strong enough to prove Theorem \ref{adjthm5}, establishing the existence of an important class of algebraic Quillen adjunctions. In order to prove the naturality component of this result, we must show that the ``unit'' functors (\ref{freediag}) constructed in Garner's small object argument satisfy a stronger universal property than was previously known. In \cite{garnerunderstanding}, Garner shows that these functors are universal among morphisms of awfs. In Section \ref{cobssec}, we show that they are universal among all adjunctions of awfs, a result that should be of independent categorical interest. In particular, it follows that two canonical methods of assigning coalgebra structures to generating cofibrations in the image of a left adjoint are the same. 

\subsection{Algebras and adjunctions}\label{algssec}

Consider an adjunction $\xymatrix{ T \colon \M \ar@<1ex>[r] \ar@{}|{\perp}[r] & \K \colon S \ar@<1ex>[l]}$ where $\J$ generates an awfs $(\CC,\FF)$ on $\M$ and $T\J$ generates an awfs $(\L,\R)$ on $\K$. If  $(\ol{\C},\F)$ is the wfs underlying $(\CC,\FF)$ and $(\ol{\cL},\cR)$ is the wfs underlying $(\L,\R)$, then $$T(\ol{\C}) \subset \ol{\cL} \hspace{.5cm} \mathrm{and} \hspace{.5cm} S(\cR) \subset \F$$ because the defining lifting properties are adjunct.

The next few sections work towards an algebraization of this result. Because the awfs are cofibrantly generated, it will be considerably easier to prove statements involving the categories of algebras. 

\begin{thm}\label{Sthm} For any adjunction $\xymatrix{ T \colon \M \ar@<1ex>[r] \ar@{}|{\perp}[r] & \K \colon S \ar@<1ex>[l]}$ where a small category $\J$ generates an awfs $(\CC,\FF)$ on $\M$ and $T\J$ generates an awfs $(\L,\R)$ on $\K$, the right adjoint $S$ lifts to a functor $\raisebox{.25in}{\xymatrix{ \Ralg \ar@{-->}[r]^{\tilde{S}} \ar[d]_U & \Falg \ar[d]^U \\ \K^{\bf 2} \ar[r]^S & \M^{\bf 2}}}$
\end{thm}
\begin{proof}
Because the awfs are cofibrantly generated, we have isomorphisms $\Ralg \cong T\J^{\boxslash}$ and $\Falg \cong \J^{\boxslash}$ that commute with the forgetful functors to the underlying arrow categories. Using the notation of the proof of Theorem \ref{adjthm}, let $(f,\psi) \in T\J^{\boxslash}$ and define $\tilde{S}(f,\psi) := (Sf, \psi^{\sharp}) \in \J^{\boxslash}$. Given $(u,v) \colon (f,\psi) \ra (g,\phi) \in T\J^{\boxslash}$,  $$\tilde{S}(u,v):=(Su, Sv) \colon (Sf, \psi^{\sharp}) \ra (Sg, \phi^{\sharp})$$ is a morphism in $\J^{\boxslash}$ by naturality of the adjunction. This defines the desired lift.
\end{proof}

There is a well-known categorical result \cite[Lemma 1]{johnstoneadjoint}, which says that lifts of $S$ to functors $\tilde{S} \colon \Ralg \ra \Falg$ are in bijective correspondence with natural transformations $\vec{\rho} \colon FS \Ra SR$ satisfying \begin{equation}\label{vecrhocond}\raisebox{.27in}{\xymatrix{ & S \ar[dl]_{\vec{\eta}_S} \ar[dr]^{S\vec{\eta}}\\ FS \ar[rr]^{\vec{\rho}} & & SR }}\hspace{.5cm} \text{and} \hspace{.5cm}  \raisebox{.55in}{\xymatrix@C=2pt{  & & FSR \ar[drr]^{\vec{\rho}_R} & & \\ FFS \ar[urr]^{F\vec{\rho}} \ar[dr]_{\vec{\mu}_S} & & & & SRR \ar[dl]^{S\vec{\mu}} \\ & FS \ar[rr]^{\vec{\rho}} & & SR &}}\end{equation} A pair $(S,\vec{\rho})$ satisfying these conditions is called a \emph{lax morphism of monads}.

Let $Q \colon \M^{\bf 2} \ra \M$ and $E\colon \K^{\bf 2} \ra \K$ be the functors accompanying the functorial factorizations of $(\CC,\FF)$ and $(\L,\R)$, respectively.  Because $F$ and $R$ are monads over cod, $\vec{\rho}= (\rho,1)$ where $\rho \colon QS \Ra SE$ is a natural transformation satisfying \begin{equation}\label{rhocond}\raisebox{.27in}{\xymatrix{ \dom S \ar[d]_{CS} \ar[r]^{SL} & SE \ar[d]^{SR} \\ QS \ar[ur]_{\rho} \ar[r]_{FS} & \cod S}}  \hspace{.5cm} \text{and} \hspace{.5cm} \raisebox{.55in}{\xymatrix@C=2pt{  & & QSR \ar[drr]^{\rho_R} & & \\ QFS \ar[urr]^{Q(\rho,1)} \ar[dr]_{\mu_S} & & & & SER \ar[dl]^{S\mu} \\ & QS \ar[rr]^{\rho} & & SE &}}\end{equation} The functor $\tilde{S}$ is defined by mapping the $\R$-algebra $(g, t \colon Eg \ra \dom\, g)$ to the $\FF$-algebra $(Sg, St\cdot \rho_g \colon QSg \ra \dom, Sg)$.

If the direction of $\rho$ is reversed, a pair $(S,(\rho,1))$ satisfying diagrams analogous to (\ref{vecrhocond}) is called a \emph{colax morphism of monads}. If the monads are replaced by their corresponding comonads and the direction of $\rho$ is unchanged, a pair $(S,(1,\rho))$ satisfying diagrams analogous to (\ref{vecrhocond}) is called a \emph{lax morphism of comonads}. If the direction of $\rho$ is reversed as well, the pair $(S,(1,\rho))$ is called a \emph{colax morphism of comonads}. This last type of morphism is in bijective correspondence with lifts of $S$ to a functor between the categories of coalgebras for the comonads by the dual of the lemma mentioned above.

For the lift $\tilde{S}$ of Theorem \ref{Sthm}, it is not easy to describe $\rho$ explicitly because we cannot easily write down the inverse to the isomorphism (\ref{algfreeiso}). Surprisingly, in light of the definitions of the next sections, this will be no great obstacle. 

\subsection{Lax morphisms and colax morphisms of awfs}\label{laxcolaxdefnssec}

The statement analogous to Theorem \ref{Sthm} for the left adjoint and categories of coalgebras is considerably harder to prove. In fact, we will prove a stronger result and deduce this as a corollary. First, we establish the relevant terminology, which extends the lax and colax morphisms of monads and comonads, introduced in the last section.

For the following definitions let $(\CC,\FF)$ be an awfs on a category $\M$ with $Q\colon \M^{\bf 2} \ra \M$ the functor accompanying its functorial factorization, and let $(\L,\R)$ be an awfs on $\K$ with $E\colon \K^{\bf 2} \ra \K$ accompanying its functorial factorization.

\begin{defn}\label{laxmordefn} A \emph{lax morphism of awfs} $(S,\rho) \colon (\L, \R) \ra (\CC,\FF)$ consists of a functor $S \colon \K \ra \M$ and a natural transformation $\rho \colon QS \Ra SE$ such that $(1,\rho) \colon CS \Ra SL$ is a lax morphism of comonads and $(\rho,1) \colon FS \Ra SR$ is a lax morphism of monads, i.e., such that the following commute:
\begin{equation}\label{laxmoreq}\raisebox{.27in}{\xymatrix{ \dom S \ar[d]_{CS} \ar[r]^{SL} & SE \ar[d]^{SR} \\ QS \ar[ur]_{\rho} \ar[r]_{FS} & \cod S}}  \raisebox{.55in}{\xymatrix@C=1pt{ & QS \ar[rr]^{\rho} \ar[dl]_{\d_S} & & SE \ar[dr]^{S\d} & \\ QCS \ar[drr]^{Q(1,\rho)} & & & & SEL \\ & & QSL \ar[urr]^{\rho_L} & &}} \raisebox{.55in}{\xymatrix@C=1pt{  & & QSR \ar[drr]^{\rho_R} & & \\ QFS \ar[urr]^{Q(\rho,1)} \ar[dr]_{\mu_S} & & & & SER \ar[dl]^{S\mu} \\ & QS \ar[rr]^{\rho} & & SE &}}\end{equation}
\end{defn}

Lax morphisms of monads $(\rho,1)$ are in bijection with lifts of $S$ to functors $\tilde{S} \colon \Ralg \ra \Falg$. Lax morphisms of comonads $(1,\rho)$ are in bijection with extensions of $S$ to functors $\hat{S} \colon \coKl(\L) \ra \coKl(\CC)$.

\begin{defn}\label{colaxmordefn} A \emph{colax morphism of awfs} $(T,\g) \colon (\CC, \FF) \ra (\L,\R)$ consists of a functor $T \colon \M \ra \K$ and a natural transformation $\g \colon TQ \Ra ET$ such that $(1,\g) \colon TC \Ra LT$ is a colax morphism of comonads and $(\g,1) \colon TF \Ra RT$ is a colax morphism of monads, i.e., such that the following commute:
\begin{equation}\label{colaxmoreq}\raisebox{.27in}{\xymatrix{ \dom T \ar[d]_{TC} \ar[r]^{LT} & ET \ar[d]^{RT}  \\ TQ \ar[ur]_{\g} \ar[r]_{ TF} &  \cod T }}  \raisebox{.55in}{\xymatrix@C=1pt{ & TQ \ar[rr]^{\gamma} \ar[dl]_{T\d} & & ET \ar[dr]^{\d_T} & \\ TQC \ar[drr]^{\gamma_C} & & & & ELT \\ & & ETC \ar[urr]^{E(1,\gamma)} & &}} \raisebox{.55in}{\xymatrix@C=1pt{  & & ETF \ar[drr]^{E(\g,1)} & & \\ TQF \ar[urr]^{\g_F} \ar[dr]_{T\mu} & & & & ERT \ar[dl]^{\mu_T} \\ & TQ \ar[rr]^{\g} & & ET &}}\end{equation}
\end{defn}

Colax morphisms of comonads $(1,\g)$ are in bijection with lifts of $T$ to functors $\tilde{T} \colon \Calg \ra \Lalg$. Colax morphisms of monads $(\g,1)$ are in bijection with extensions of $T$ to functors $\hat{T} \colon \Kl(\FF) \ra \Kl(\R)$. 

\begin{ex} A morphism of awfs $\rho \colon (\CC,\FF) \ra (\L,\R)$, defined in \ref{morawfsdefn}, is simultaneously a lax morphism of awfs $(1,\rho) \colon (\L,\R) \ra (\CC,\FF)$ and a colax morphism of awfs $(1,\rho) \colon (\CC,\FF) \ra (\L,\R)$. Conversely, every lax or colax morphism of awfs over an identity functor is a morphism of awfs.
\end{ex}

Lax and colax morphisms of awfs can be identified by the following recognition principle, which extends the material of Section \ref{compssec}.

\begin{lem}\label{laxawfslem} Suppose $(\L,\R)$ and $(\CC,\FF)$ are awfs and $(S,\vec{\rho}) \colon \R \ra \FF$ is a lax morphism of monads corresponding to a lift $\tilde{S} \colon \Ralg \ra \Falg$ of $S$. Then $(S,\rho) \colon (\L,\R) \ra (\CC,\FF)$ is a lax morphism of awfs if and only if $\tilde{S}$ preserves the canonical composition of algebras. Dually, a colax morphism between the comonads of awfs is a colax morphism of awfs if and only if the lifted functor preserves the canonical composition of coalgebras.
\end{lem}
\begin{proof}
Suppose $(S,\rho)$ is a lax morphism of awfs and let $(f,s)$ and $(g,t)$ be composable $\R$-algebras. By definition \begin{align*}\tilde{S}(g,t) \bullet \tilde{S}(f,s) &= (Sg \cdot Sf, (St \cdot \rho_g) \bullet (Ss \cdot \rho_f)) \\ &= (S(gf), Ss \cdot \rho_f \cdot Q(1, St \cdot \rho_g \cdot Q(Sf,1)) \cdot \d_{S(gf)}), \intertext{while} \tilde{S}((g,t) \bullet(f,s)) &= \tilde{S} (gf, s \cdot E(1,t\cdot E(f,1)) \cdot \d_{gf}) \\ &= (S(gf), Ss \cdot SE(1,t \cdot E(f,1)) \cdot S\d_{gf} \cdot \rho_{gf})  .\end{align*}
The diagram
$$\xymatrix{ & SEgf \ar[rr]^{S\d_{gf}} & & SELgf \ar[dr]^{SE(1,t\cdot E(f,1))} & \\ QSgf \ar[ur]^{\rho_{gf}} \ar[dr]_{\d_{S(gf)}} & & QSLgf \ar[ur]_{\rho_{Lgf}} \ar[dr]^{Q(1,St \cdot SE(f,1))} & & SEf \\ & QCSgf \ar[rr]_{Q(1,St \cdot SE(f,1) \cdot \rho_{gf})} \ar[ur]^{Q(1,\rho_{gf})} & & QSf \ar[ur]_{\rho_f} & }$$
which commutes by (\ref{laxmoreq}) and naturality of $\rho$ shows that both algebra structures are the same.

Conversely, we must show that the center diagram of $(\ref{laxmoreq})$ commutes if the functor $\tilde{S}$ defined via $\rho$ preserves composition of algebras. The proof requires a straightforward diagram chase:
\begin{align*} S\d \cdot \rho &= S(\mu \bullet \mu_L) \cdot SE(L^2,1) \cdot \rho & \text{definition of $\d$}\\ &= S (\mu \bullet \mu_L) \cdot \rho_{R\cdot RL} \cdot Q(SL^2,1) & \text{naturality of $\rho$} \\ &= ((S\mu \cdot \rho_R) \bullet (S \mu_L \cdot \rho_{RL})) \cdot Q(SL^2,1) & \text{$\tilde{S}$ preserves composition} \\ &= S\mu_L\cdot \rho_{RL} \cdot Q(1,S\mu) \cdot Q(1,\rho_R)& \\ &\quad\cdot Q(1,Q(SRL,1)) \cdot \d_{SR \cdot SRL} \cdot Q(SL^2,1) & \text{defn.~of comp.~in $\Falg$} \\ &= S\mu_L \cdot \rho_{RL} \cdot Q(1,S\mu) \cdot Q(1,\rho_R) & \\ &\quad \cdot Q(1,Q(SL,1)) \cdot Q(SL^2,1) \cdot \d_S & \text{nat.~of $\d$; functoriality of $Q$} \\ &= S\mu_L \cdot \rho_{RL} \cdot Q(1,S\mu) \cdot Q(1,SE(L,1)) & \\ &\quad  \cdot Q(SL^2,1) \cdot  Q(1,\rho) \cdot \d_S &\text{nat.~of $\rho$; functoriality of $Q$} \\ &= S\mu_L \cdot \rho_{RL} \cdot Q(SL^2,1) \cdot Q(1,\rho) \cdot \d_S &\text{monad triangle identity} \\ &= S\mu_L \cdot SE(L^2,1) \cdot \rho_L \cdot Q(1,\rho) \cdot \d_S & \text{naturality of $\rho$} \\ &= \rho_L \cdot  Q(1,\rho) \cdot \d_S & \text{monad triangle identity}\ \qedhere
\end{align*}
\end{proof}

\subsection{Adjunctions of awfs}\label{adjdefnssec}

The notions of lax and colax morphisms of awfs are closely related. In fact, given a lax morphism of awfs $(S,\rho) \colon (\L,\R) \ra (\CC,\FF)$ and an adjunction $(T,S,\iota,\nu)$ where $T \dashv S$ and $\iota$ and $\nu$ are the unit and counit, there is a canonical natural transformation $\g \colon TQ \Ra ET$ such that $(T,\g) \colon (\CC,\FF) \ra (\L,\R)$ is a colax morphism of awfs. The dual result holds as well. Combining the data of the corresponding lax and colax morphisms of awfs we obtain the concept of an \emph{adjunction of awfs}, defined below. But first, we need the following categorical concept to explain the relationship between $\rho$ and $\g$.

Given functors as in the diagram $$\xymatrix{\mathcal{A} \ar@<-1ex>[d]_-F \ar@{}[d]|-{\dashv} \ar[r]^I & \mathcal{C} \ar@<-1ex>[d]_-H \ar@{}[d]|-{\dashv} \\ \mathcal{B} \ar@<-1ex>[u]_-G \ar[r]_J & \mathcal{D} \ar@<-1ex>[u]_-K}$$ with $\eta$ and $\e$ the unit and counit for $F \dashv G$ and $\iota$ and $\nu$ the unit and counit for $H \dashv K$, there is a bijection between natural transformations $$\raisebox{.25in}{\xymatrix{ \mathcal{A} \ar[d]_F \ar[r]^I & \mathcal{C} \ar@{=>}[dl]_{\a} \ar[d]^H  \\   \mathcal{B} \ar[r]_J & \mathcal{D} }} \hspace{.5cm} \text{and} \hspace{.5cm} \raisebox{.25in}{\xymatrix{ \mathcal{A} \ar[r]^I \ar@{=>}[dr]^{\beta} & \mathcal{C} \\ \mathcal{B} \ar[r]_J \ar[u]^G & \mathcal{D} \ar[u]_K }}$$ given by the formulae $$\b = KJ \e \cdot K \a_G \cdot \iota_{IG} \hspace{.5cm} \text{and} \hspace{.5cm} \a = \nu_{JF} \cdot H \beta_F \cdot HI \eta.$$ The corresponding natural transformations $\a$ and $\b$ are called \emph{mates} \cite{kellystreetreview}.

\begin{defn}\label{adjawfsdefn2} An \emph{adjunction of awfs} $(T,S,\g,\rho) \colon (\CC,\FF) \ra (\L,\R)$ consists of an adjoint pair of functors $T \dashv S$ together with mates $\g$ and $\rho$ such that $(T,\g) \colon (\CC,\FF) \ra (\L,\R)$ is a colax morphism of awfs and $(S,\rho) \colon (\L,\R) \ra (\CC,\FF)$ is a lax morphism of awfs.
\end{defn}

The natural transformations $\g$ and $\rho$ should be mates with respect to the functors 
\begin{equation}\label{thesemates}\xymatrix{\M^{\bf 2} \ar@<-1ex>[d]_-T \ar@{}[d]|-{\dashv} \ar[r]^Q & \M \ar@<-1ex>[d]_-T \ar@{}[d]|-{\dashv} \\ \K^{\bf 2} \ar@<-1ex>[u]_-S \ar[r]_E & \K \ar@<-1ex>[u]_-S}\end{equation} 

As alluded to above, the criteria on $\g$ and $\rho$ in Definition \ref{adjawfsdefn2} are overdetermined:

\begin{lem}\label{odlem} Suppose we have an adjunction $\xymatrix{ (T,S,\iota,\nu) \colon \M \ar@<.5ex>[r] & \K \ar@<.5ex>[l]}$ where $\M$ has an awfs $(\CC,\FF)$ and $\K$ has an awfs $(\L,\R)$. Let $\g$ and $\rho$ be mates with respect to the functors of \emph{(\ref{thesemates})}. Then $(S,\rho) \colon (\L,\R) \ra (\CC,\FF)$ is a lax morphism of awfs if and only if $(T,\g) \colon (\CC,\FF) \ra (\L,\R)$ is a colax morphism of awfs, in which case $(T,S,\g, \rho)$ is an adjunction of awfs.
\end{lem}
\begin{proof}
Each diagram of Definition \ref{laxmordefn} is satisfied by $\rho$ if and only if its mate $\g$ satisfies the corresponding diagram of Definition \ref{colaxmordefn}, as can be verified by a diagram chase. Or see \cite{kellydoctrinal}.
\end{proof}

\begin{rmk}\label{odrmk} The proof of Lemma \ref{odlem} implies a conclusion slightly stronger than the statement. Given mates $\rho$ and $\g$ as above with respect to $T \dashv S$, to show that $(T,S,\g,\rho)$ is an adjunction of awfs, it suffices to show that either $\rho$ is a lax or $\g$ is a colax morphism of monads and that either $\rho$ is a lax or $\g$ is a colax morphism of comonads.
\end{rmk}

\begin{ex}\label{trivadjex} In any category with an algebraic model structure, the comparison map $\xi \colon (\CC_t,\FF) \ra (\CC,\FF_t)$ specifies an adjunction of awfs, where the adjunction is the trivial one with functors, unit, and counit all identities. In this case, $\xi$ is its own mate and the requirements that $(1,\xi) \colon (\CC_t, \FF) \ra (\CC,\FF_t)$ is a colax morphism of awfs and that $(1,\xi) \colon (\CC,\FF_t) \ra (\CC_t,\FF)$ is a lax morphism of awfs are equivalent.
\end{ex}

Less trivially, there exists a canonical adjunction of awfs $(\CC,\FF) \ra (\L,\R)$ whenever $(\CC,\FF)$ is generated by $\J$ and $(\L,\R)$ is generated by $T\J$ for some adjunction $T \dashv S$.  By Theorem \ref{Sthm}, there exists a natural transformation $\rho \colon QS \Ra SE$ such that $({\rho},1)$ is a lax morphism of monads. We will show that $(S,\rho)  \colon (\L,\R) \ra (\CC,\FF)$ is a lax morphism of awfs. It follows from Lemma \ref{odlem} that this situation gives rise to an adjunction of awfs $(T,S,\g,\rho)$, where $\g$ is the mate of $\rho$, proving the following theorem.

\begin{thm}\label{adjawfsthm}
Consider an adjunction $\xymatrix{T \colon \M \ar@<1ex>[r] \ar@{}|{\perp}[r] & \K \colon S \ar@<1ex>[l]}$ where $\J$ generates an awfs $(\CC,\FF)$ on $\M$ and $T\J$ generates an awfs $(\L,\R)$ on $\K$. Let $\rho$ be the natural transformation determined by Theorem \ref{Sthm} and let $\g$ be its mate. Then $(T,S,\g, \rho) \colon (\CC,\FF) \rightarrow (\L,\R)$ is an adjunction of awfs.
\end{thm}

First, we must show that $(S,\rho)$ is a lax morphism of awfs. Doing so directly is possible but quite hard, because a concrete understanding of $\rho$ is only obtained by laboriously computing $\g$, vis-\`{a}-vis running through the details of the small object argument. We use Lemma \ref{laxawfslem} instead.

\begin{proof}[Proof of Theorem \ref{adjawfsthm}]
By Lemma \ref{laxawfslem}, it suffices to show that the functor $\tilde{S}$ defined in the proof of Theorem \ref{Sthm} preserves the canonical composition of algebras. Then Lemma \ref{odlem} implies that $(T,S,\g,\rho)$ is an adjunction of awfs, where $\g$ is the mate of $\rho$.

Suppose $(f,\phi)$ and $(g,\psi)$ are composable objects of $T\J^{\boxslash}$. By definition, the functor $\tilde{S}$ takes the composite, described in Example \ref{compex}, to the morphism $S(gf)$ with the lifting function: $$(\psi\bullet\phi)^{\sharp}(j,a,b) = S \phi (Tj, \nu \cdot Ta, \psi(Tj, f \cdot \nu \cdot Ta, \nu \cdot Tb)) \cdot \iota$$ where $\iota$ and $\nu$ are the unit and counit of $T \dashv S$. By contrast \begin{align*} \psi^{\sharp} \bullet \phi^{\sharp}(j,a,b) &= \phi^{\sharp}(j, a, \psi^{\sharp}(j, Sf \cdot a, b)) \\ &= S \phi(Tj, \nu \cdot Ta, \nu \cdot TS \psi(Tj, \nu \cdot TSf \cdot Ta, \nu \cdot Tb) \cdot T\iota)\cdot \iota \\ &= S \phi(Tj, \nu \cdot Ta, \psi(Tj, f\cdot \nu \cdot Ta, \nu\cdot Tb)\cdot \nu_T \cdot T\iota )\cdot \iota, \end{align*} which is the same as above, after application of a triangle identity for the adjunction $T \dashv S$.
\end{proof}

Theorem \ref{adjawfsthm} extends Theorem \ref{Sthm} and the corresponding result for coalgebras, which we note, for completeness sake, as an immediate corollary.

\begin{cor}\label{Tthm} For any adjunction $\xymatrix{ T \colon \M \ar@<1ex>[r] \ar@{}|{\perp}[r] & \K \colon S \ar@<1ex>[l]}$ where a small category $\J$ generates an awfs $(\CC,\FF)$ on $\M$ and $T\J$ generates an awfs $(\L,\R)$ on $\K$, the left adjoint $T$ lifts to a functor $\raisebox{.25in}{\xymatrix{ \Calg \ar@{-->}[r]^{\tilde{T}} \ar[d]_U & \Lalg \ar[d]^U \\ \M^{\bf 2} \ar[r]^T & \K^{\bf 2}}}$
\end{cor}

It is an easy exercise to check that adjunctions of awfs are composable, i.e., given adjunctions of awfs $(\CC,\FF) \rightarrow (\L',\R')$ and $(\L',\R') \rightarrow (\L,\R)$, the composite adjoint pair of functors and pasted natural transformations form an adjunction of awfs $(\CC,\FF) \rightarrow (\L,\R)$. Hence, the following corollary combines Example \ref{trivadjex} and Theorem \ref{adjawfsthm} to find adjunctions of awfs in a weaker situation, opening up an array of potential examples.

\begin{cor}\label{adjawfscor} Suppose we have an adjunction $\xymatrix{ T \colon \M \ar@<1ex>[r] \ar@{}|{\perp}[r] & \K \colon S \ar@<1ex>[l]}$ where $\J$ generates an awfs $(\CC,\FF)$ on $\M$ and $\K$ has an awfs $(\L,\R)$, not necessarily cofibrantly generated. Suppose also that $\K$ permits the small object argument and that we have a functor $\J \ra \Lalg$ lifting $T$. Then $T$ and $S$ give rise to an adjunction of awfs $(\CC,\FF) \rightarrow (\L,\R)$.
\end{cor}
\begin{proof}
By Theorem \ref{cgthm}, there exists an awfs $(\L',\R')$ on $\K$ that is cofibrantly generated by $T\J$. By Theorem \ref{adjawfsthm}, $T$ and $S$ give rise to an adjunction of awfs $(\CC,\FF) \rightarrow (\L',\R')$. The functor $\J \ra \Lalg$ lifting $T$ is equivalently described as a functor $T\J \ra \Lalg$ over $\K^{\bf 2}$. By the universal property of $T\J \ra {\mathbb{L}'\text{-}\mathrm{{\bf coalg}}}$, there exists a morphism of awfs $(\L', \R') \ra (\L,\R)$, which is equivalently an adjunction of awfs $(\L',\R') \rightarrow (\L,\R)$ over the identity functors on $\K^{\bf 2}$. We obtain the desired adjunction of awfs $(\CC,\FF) \rightarrow (\L,\R)$ by composing the above two.
\end{proof}

\subsection{Change of base in Garner's small object argument}\label{cobssec}

In the situation of Theorem \ref{adjawfsthm}, there are two canonical methods for assigning $\L$-coalgebra structures to the objects $Tj$ of the generating category $T\J$. One method applies the functor $T$ to the canonical $\CC$-coalgebra structure for $j$ and then composes with the natural transformation $\g$ accompanying the lift $\tilde{T} \colon \Calg \ra \Lalg$. The other simply assigns $Tj$ the canonical coalgebra structure given by Garner's small object argument via the functor (\ref{freediag}). We might hope that the two results are the same. This is the content of an immediate corollary to the main theorem of this section.

\begin{cor}\label{natunitcor} Given an adjunction $T \dashv S$ between categories $\M$ and $\K$,  consider a category $\J$ which generates an awfs $(\CC,\FF)$ on $\M$ and such that $T\J$ generates an awfs $(\L,\R)$ on $\K$. Then the functor $\tilde{T}$ arising from the canonical adjunction of awfs $(\CC,\FF) \ra (\L,\R)$ commutes with the units exhibiting the ``freeness'' of cofibrantly generated awfs, i.e., the diagram
\begin{equation}\label{cobunit}\xymatrix@R=15pt@C=10pt{ \J \ar[dd] \ar[dr]^{\l^{\M}} \ar@/^/[drrr]^{\l^{\K}} \\&   \Calg \ar[rr]^-{\tilde{T}} \ar[dl] & & \Lalg \ar[dl] \\  \M^{\bf 2} \ar[rr]_-{T} & & \K^{\bf 2}} \end{equation}
commutes.
\end{cor} 

Given a category $\M$ that permits the small object argument, Garner's construction produces a reflection of any small category $\J$ over $\M^{\bf 2}$ along the so-called ``semantics'' functor \begin{equation}\label{gsemantics}{\small \xymatrix@R=2pt{\G= \AWFS(\M) \ar[r]^-{\G_1} & \LAWFS(\M) \ar[r]^-{\G_2} & \Cmd(\M^{\bf 2}) \ar[r]^-{\G_3} & \CAT/\M^{\bf 2} \\ (\CC,\FF) \ar@{|->}[r] & (\CC,Q) \ar@{|->}[r] & \CC \ar@{|->}[r] & \Calg}}\end{equation}
from the category of awfs on $\M$ and morphisms of awfs to the slice category over $\M^{\bf 2}$ \cite[\S 4]{garnerunderstanding}. Here, $\Cmd(\M^{\bf 2})$ is the category of comonads on $\M^{\bf 2}$ and comonad morphisms and $\LAWFS(\M)$ is the full subcategory of comonads over dom, or equivalently the category of functorial factorizations, whose left functor is a comonad. 

The component of the unit of this reflection at a small category $\J$ generating an awfs $(\CC,\FF)$ is the functor $\J \ra \Calg$ over $\M^{\bf 2}$ of (\ref{freediag}), which is universal with respect to morphisms of awfs. We prove that these maps are universal with respect to all adjunctions of awfs. To find an appropriate categorical context for the statement and proof of this result, we must enlarge the categories of (\ref{gsemantics}). The new domain is the category $\AWFS_{\ladj}$, whose objects are awfs and whose morphisms are adjunctions of awfs. In analogy with (\ref{gsemantics}), there is a forgetful functor to $\CAT/(-)^{\bf 2}_{\ladj}$, whose objects are categories sliced over arrow categories, whose morphisms are adjunctions between the base categories (before applying $(-)^{\bf 2}$) together with a chosen lift of the left adjoint to the fibers. This ``semantics'' functor factors as:
\begin{equation}\label{semantics}{\small \xymatrix@C=17pt{\G^{\ladj}= \AWFS_{\ladj} \ar[r]^-{\G^{\ladj}_1} & \LAWFS_{\ladj} \ar[r]^-{\G^{\ladj}_2} & \Cmd(-)^{\bf 2}_{\ladj} \ar[r]^-{\G^{\ladj}_3} & \CAT/(-)^{\bf 2}_{\ladj}}}\end{equation} Here, $\Cmd(-)^{\bf 2}_{\ladj}$ is the category of comonads on arrow categories and colax morphisms of comonads whose functor is the left adjoint of a specified adjunction, and $\LAWFS_{\ladj}$ is the full subcategory of comonads over dom.

When restricted to objects whose base categories are cocomplete, each category in (\ref{semantics}) is cofibered over $\CAT^{\bf 2}_{\ladj}$, the category of arrow categories and adjunctions of underlying categories, regarded as morphisms in the direction of the left adjoint. The fibers over the identity arrows are exactly the categories of (\ref{gsemantics}). For $\AWFS_{\ladj}$, we perhaps need to insist that the categories be locally finitely presentable, in which case this statement says that any awfs can be lifted along an adjunction, even if it is not cofibrantly generated. This decidedly non-trivial result is due to Richard Garner.

In \cite{garnerunderstanding}, Garner shows that when $\M$ permits the small object argument, any small category can be reflected along (\ref{gsemantics}). We show that this construction gives a reflection of these objects along (\ref{semantics}), which is precisely what is needed for the desired corollary.

\begin{thm}\label{cobthm} For any small category $\J$ over the arrow category of a category $\M$ that permits the small object argument, the unit functor constructed by Garner's small object argument is universal among adjunctions of awfs.
\end{thm}
\begin{proof}
It suffices to show that such $\J$ can be reflected along each of the $\G^{\ladj}_i$, i.e., the unit functor constructed at each step in \cite[\S 4]{garnerunderstanding} satisfies the appropriate universal property.

The reflection of $\J$ along $\G^{\ladj}_3$ is its density comonad $\CC^0$, i.e., the left Kan extension of $J \colon \J \ra \M^{\bf 2}$ along itself. If $T \colon \M \ra \K$ is a left adjoint and $\L$ is an arbitrary comonad on $\K^{\bf 2}$, functors $\J \ra \Lalg$ lifting $T$ are in bijection with natural transformations $TJ \Ra LTJ$, which are in bijection with natural transformations $TC^0 \Ra LT$ because left adjoints preserve left Kan extensions. By the universal property of the density comonad, or alternatively, by a straightforward diagram chase, such natural transformations are always comonad morphisms \cite[Chapter II]{dubuckan}. This shows that the unit $\J \ra \CC^0$-{\bf coalg} is universal with respect to comonad morphisms lifting left adjoints and hence gives a reflection of $\J$ along $\G^{\ladj}_3$.

The refection of $\CC^0$ along $\G^{\ladj}_2$ is given by an ofs (see Example \ref{ofsex}) on arrow categories, which factors a given morphism (a square in the underlying category) as a pushout square followed by a square whose domain component is an identity. Explicitly, the reflection $\CC^1$ is obtained by factoring the counit of $\CC^0$ as depicted below: \begin{equation*}\label{G2eq}\xymatrix{ \cdot \ar[d]_{C^0f} \ar[r] \ar@{}[dr]|(.8){\ulcorner} & \cdot \ar[d]^{C^1f} \ar@{=}[r] & \cdot \ar[d]^f \\ \cdot \ar[r] & \cdot \ar[r] & \cdot}\end{equation*}

Let $\psi \colon C^0 \Ra C^1$ be the natural transformation whose component at $f$ is the left-hand square depicted above. Given a colax morphism of comonads $(T\colon \M \ra \K, \g\colon TC^0 \Ra LT)$ where $T$ is a left adjoint and $\L$ is a comonad on $\K^{\bf 2}$ over $\dom \colon \K^{\bf 2} \ra \K$, we have a commutative square $$\xymatrix{ TC^0 \ar[r]^{\g} \ar[d]_{T \psi} & LT \ar[d]^{\e T} \\ TC^1 \ar[r]_{T\e} & T}$$ because $\g$ is a comonad morphism and the lower left composite is $T$ applied to the counit of $C^0$. The left arrow $T\psi$ is in the left class of the ofs described above because $T$, as a left adjoint between the underlying categories, preserves pushouts, and so the components of $T\psi$ are pushout squares. The right arrow is in the right class because $\L$ was assumed to be a comonad over dom. The ofs described above, this time on $\K$, solves the lifting problem to obtain the components of a unique natural transformation $\g' \colon TC^1 \Ra LT$. By setting up appropriate lifting problems and using the fact that any solutions that exist must be unique, we can easily check that $\g'$ is a comonad morphism, as desired. This shows that the unit functor $\CC^0$-{\bf coalg}$\ra \CC^1$-{\bf coalg} is universal with respect to colax morphisms of comonads over $\dom$ that lift a left adjoint; hence, it exhibits $\CC^1$ as the reflection of $\CC^0$ along $\G^{\ladj}_2$.

It remains only to consider the reflection of $\CC^1$ along $\G^{\ladj}_1$. For each category $\M$, there is a strict two-fold monoidal category $\FunF(\M)$ of functorial factorizations of $\M$ (see \cite{garnercofibrantly, garnerunderstanding} and \cite{bfsviterated}), for which $\LAWFS(\M)$ is the category of $\odot$-comonoids and $\AWFS(\M)$ is the category of $\otimes, \odot$-bialgebras. The product $\otimes$ (resp.~$\odot$) uses the second awfs to re-factor the right (resp.~left) half of the factorization produced by the first awfs. To reflect from $\odot$-comonoids into bialgebras, Garner uses Max Kelly's construction of the free $\otimes$-monoid on a pointed object \cite{kellyunified}, in this case the unique arrow $I \ra \vec{Q^1}$ from the unit for $\otimes$, which is initial in $\FunF(\M)$, to  the functorial factorization of $\CC^1$. 

Let $\FunF_{\ladj}$ be the category of functorial factorizations over an arbitrary base whose morphisms are \emph{colax morphisms of functorial factorizations} lifting left adjoints. If $\vec{X}$ is a functorial factorization on $\M$ and $\vec{Y}$ is a functorial factorization on $\K$, then a morphism $\phi\colon \vec{X} \ra \vec{Y}$ lifting a left adjoint $T\colon \M \ra \K$ is a natural transformation $\phi\colon TX \Ra YT$ such that the two triangles analogous to the left-hand diagram of (\ref{colaxmoreq}) commute. This category is not two-fold monoidal, as we have no way to combine objects in different fibers. However, given objects $\vec{X}$ and $\vec{Z}$ in the fiber over $\M$ and $\vec{Y}$ and $\vec{W}$ in the fiber over $\K$ together with morphisms $\phi\colon \vec{X} \ra \vec{Y}$ and $\psi\colon \vec{Z} \ra \vec{W}$ lifting the same left adjoint $T\colon \M \ra \K$, we do obtain lifts $\phi \otimes \psi \colon \vec{X} \otimes \vec{Z} \ra \vec{Y} \otimes \vec{W}$ and $\phi \odot \psi \colon \vec{X} \odot \vec{Z} \ra \vec{Y} \odot \vec{W}$ of $T$.

Furthermore, if $\phi$ and $\psi$ are $\odot$-comonoid morphisms, then so is $\phi \otimes \psi$. The proof uses the fact that $\odot$ distributes over $\otimes$ in each fiber \cite[\S 3.2]{garnercofibrantly}, and the canonical arrows $\a$ exhibiting this distributivity are natural with respect to colax morphisms of functorial factorizations: 
$${ \xymatrix{ (\vec{X} \odot \vec{X'}) \otimes (\vec{Z} \odot \vec{Z'}) \ar[r]^{\a} \ar[d]|{(\phi \odot \phi') \otimes (\psi \odot \psi')} & (\vec{X} \otimes \vec{Z}) \odot (\vec{X'} \otimes \vec{Z'}) \ar[d]|{(\phi \otimes \psi) \odot (\phi' \otimes \psi')} \\ (\vec{Y} \odot \vec{Y'}) \otimes (\vec{W} \odot \vec{W'}) \ar[r]_{\a} & (\vec{Y} \otimes \vec{W} ) \odot (\vec{Y'} \otimes \vec{W'})}}$$ In other words, if $\phi$ and $\psi$ are morphisms in $\LAWFS_{\ladj}$, so is $\phi \otimes \psi$. This is all the structure we need to prove the unit satisfies the desired universal property; we need not consider the inner workings of the category $\LAWFS_{\ladj}$ any further.

Given a morphism from a pointed object $I \ra \vec{X}$ in the fiber over $\M$ to a $\otimes$-monoid $\vec{Y}$ in the fiber over $\K$, we inductively obtain morphisms from the colimits involved in Kelly's transfinite construction to $\vec{Y} \otimes \vec{Y}$ and thus to $\vec{Y}$ by applying the multiplication $\mu \colon \vec{Y} \otimes \vec{Y} \ra \vec{Y}$. By the universal property, the resulting morphism from the free $\otimes$-monoid on $I \ra \vec{X}$ to $\vec{Y}$ is a $\otimes$-monoid morphism in the category of $\odot$-comonoids and $\odot$-comonoid morphisms, and is unique. Applying this to the situation at hand, a colax morphism of comonads $(\CC^1,Q^1) \ra (\L, E)$ lifting the left adjoint of a specified adjunction and whose target underlies an awfs $(\L,\R)$ factors through a unique colax morphism of awfs $(\CC,\FF) \ra (\L,\R)$. By Lemma \ref{odlem}, this determines a unique adjunction of awfs. Hence, the unit of this reflection satisfies the desired universal property, completing the proof.
\end{proof}

The desired corollary now follows immediately from the universal property of $\lambda^{\M}$. We will need this result in the next section.

\section{Algebraic Quillen adjunctions}\label{algquillensec}

We can now prove that the adjunction between the algebraic model structures of Theorem \ref{adjthm} is canonically an algebraic Quillen adjunction.

Recall the following definition.

\begin{defn1} Let $\M$ have an algebraic model structure $\xi^{\M} \colon (\CC_t,\FF) \ra (\CC,\FF_t)$ and let $\K$ have an algebraic model structure $\xi^{\K} \colon (\L_t,\R) \ra (\L, \R_t)$. An adjunction $\xymatrix{ T \colon \M \ar@<1ex>[r] \ar@{}[r]|-{\perp} & \K \colon S \ar@<1ex>[l]}$ is an \emph{algebraic Quillen adjunction} if there exist natural transformations $\g_t$, $\g$, $\rho_t$, and $\rho$ determining five adjunctions of awfs
\begin{equation}\label{algquilldiag}\xymatrix@R=30pt@C=80pt{  (\CC_t, \FF) \ar[dr]|-{(T,S,\g\cdot T\xi^{\M}, S\xi^{\K}\cdot \rho)} \ar[r]^{(T,S, \g_t,\rho)} \ar[d]_{(1,1,\xi^{\M},\xi^{\M})} & (\L_t,\R) \ar[d]^{(1,1,\xi^{\K},\xi^{\K})} \\ (\CC,\FF_t) \ar[r]_{(T,S,\g,\rho_t)}   & (\L,\R_t) }\end{equation} such that both triangles commute.
\end{defn1}

\begin{thm1} Let $\xymatrix{ T \colon \M \ar@<1ex>[r] \ar@{}[r]|-{\perp} & \K \colon S \ar@<1ex>[l]}$ be an adjunction. Suppose $\M$ has an algebraic model structure, generated by $\I$ and $\J$, with comparison map $\xi^{\M}$. Suppose $\K$ has the algebraic model structure, generated by $T\I$ and $T\J$, with canonical comparison map $\xi^{\K}$. Then $T \dashv S$ is canonically an algebraic Quillen adjunction. 
\end{thm1}
\begin{proof} 
Write $Q_t$, $Q$, $E_t$, and $E$ for the functors accompanying the functorial factorizations of the awfs $(\CC_t,\FF)$, $(\CC,\FF_t)$, $(\L_t,\R)$, and $(\L,\R_t)$, respectively. Then by Theorem \ref{adjawfsthm} the natural transformations $$\g_t \colon TQ_t \Ra E_t T,\  \g \colon TQ \Ra ET,\ \rho \colon Q_t S \Ra SE_t,\  \text{and}\ \rho_t \colon QS \Ra SE$$ arising from the canonical lifts of $S$ give rise to adjunctions of awfs $$(T,S,\g_t,\rho) \colon (\CC_t,\FF) \ra (\L_t,\R) \quad \text{and} \quad (T,S, \g, \rho_t) \colon (\CC,\FF_t) \ra (\L,\R_t).$$ Composing the left-hand adjunction with $\xi^{\K}$ and the right-hand adjunction with $\xi^{\M}$, which we saw in Example \ref{trivadjex} are themselves adjunctions of awfs, we obtain two canonical adjunctions of awfs \begin{equation}\label{adjawfspair}\xymatrix{(\CC_t,\FF) \ar@<.5ex>[rrr]^-{(T,S, \xi^{K}T\cdot \g_t, S\xi^{\K} \cdot \rho)} \ar@<-.5ex>[rrr]_-{(T,S, \g \cdot T\xi^{\M}, \rho_t \cdot \xi^{\M}S)} & & & (\L,\R_t).}\end{equation} We'll show that the corresponding natural transformations are the same.

By the correspondence between colax morphisms of comonads and natural transformations \cite{johnstoneadjoint}, to show that  \begin{equation}\label{natofnat} \xymatrix{ TQ_t  \ar[r]^{\g_t} \ar[d]_{T\xi^{\M}} & E_tT \ar[d]^{\xi^{\K}T} \\ TQ \ar[r]_{\g} & ET} \end{equation} commutes, it suffices to show that both composites correspond to the same lift of $T$ to a functor $\Ctalg \ra \Lalg$. 

Let $\tilde{T}_t \colon \Ctalg \ra \Ltalg$ and $\tilde{T} \colon \Calg \ra \Lalg$ denote the lifts of $T$ corresponding to $\g_t$ and $\g$, respectively. We must show the the right-hand diagram of (\ref{natliftdiag}) commutes. By the definition of $\xi^{\K}$ in the proof of Theorem \ref{adjthm}, the outer rectangle of $$\xymatrix{\J \ar[d] \ar[r]^-{\l^{\M}} & \Ctalg \ar[r]^{(\xi^{\M})_*} \ar[d]_{\tilde{T}_t} & \Calg \ar[d]^{\tilde{T}} \\ T\J \ar[r]_-{\l^{\K}} & \Ltalg \ar[r]_{(\xi^{\K})_*} & \Lalg}$$ commutes. By Corollary \ref{natunitcor}, the left-hand square commutes. By Theorem \ref{cobthm}, the unit $\l^{\M}$ is universal among adjunctions of awfs, which implies that the right-hand square commutes.

The other half of the proof now follows formally, using the fact that the natural transformations $\rho$ and $\rho_t$ defining the lifts $\tilde{S}$ and $\tilde{S}_t$ of (\ref{natliftdiag}) are mates of the natural transformations $\g_t$ and $\g$ defining the lifts $\tilde{T}_t$ and $\tilde{T}$. If $\iota$ and $\nu$ are the unit and counit of $T \dashv S$, the commutative diagram \begin{equation}\label{natofnat2}\xymatrix@C=40pt{ Q_tS \ar[d]_{\xi^{\M}S} \ar[r]^-{\iota Q_tS} & STQ_t S \ar[d]_{ST\xi^{\M}{S}} \ar[r]^{S\g_t S} & SE_t TS \ar[d]^{S \xi^{\K}TS} \ar[r]^{{SE_t(\nu,\nu)}} & SE_t \ar[d]^{S\xi^{\K}} \\ QS \ar[r]_-{\iota QS} & STQS \ar[r]_{S\g S} & SETS \ar[r]_{SE(\nu,\nu)} & SE}\end{equation} says that $S\xi^{\K}\cdot \rho = \rho_t \cdot \xi^{\M}S$. This tells us that the diagram of functors on the left-hand side of (\ref{natliftdiag}) commutes, which proves that the two adjunctions (\ref{adjawfspair}) are the same and that $T \dashv S$ is an algebraic Quillen adjunction.
\end{proof}

Note that a diagram like (\ref{natofnat2}), which shows that the natural transformations $\g \cdot T\xi^{\M} \colon TQ_t \Ra ET$ and $S\xi^{\K} \cdot \rho \colon Q_t S \Ra SE$ are mates, appears in the proof that adjunctions of awfs can be composed. In light of Corollary \ref{adjawfscor}, we expect that many other naturally occurring examples of Quillen adjunctions can be algebraicized to give algebraic Quillen adjunctions.


\begin{thebibliography}{BFSV03}
 
\bibitem[BFSV03]{bfsviterated} {\scshape Balteanu, C.; Fiedorowicz, Z.; Schw\"{a}nzel, R.; Vogt, R.} Iterated monoidal categories. {\em Adv. Math.} {\bf 176(2)} (2003) 277-349. MR1982884 (2004h:18004), Zbl 1030.18006.

\bibitem[Dub70]{dubuckan}{\scshape Dubuc, E.J.} Kan extensions in enriched category theory. Lect. Notes in Math., 145. {\em Springer-Verlag, Berlin}, 1970.  MR0280560 (43 \#6280), Zbl 0228.18002.

\bibitem[Gar07]{garnercofibrantly}{\scshape Garner, R.} Cofibrantly generated natural weak factorisation systems. Preprint, arXiv:0702.290v1 [math.CT] 2007.

\bibitem[Gar09]{garnerunderstanding}{\scshape Garner, R.} Understanding the small object argument. {\em Appl. Categ. Structures.} {\bf 17(3)} (2009) 247--285. MR2506256 (2010a:18005), Zbl 1173.55009.

\bibitem[Gar10]{garnerhomomorphisms}{\scshape Garner, R.} Homomorphisms of higher categories. {\em Adv. Math.} {\bf 224(6)} (2010), 2269--2311. MR2652207, Zbl 1205.18004.

\bibitem[GT06]{gtnatural}{\scshape Grandis, M.; Tholen, W.} Natural weak factorization systems. {\em Arch. Math. (Brno)} {\bf 42(4)} (2006) 397--408. MR2283020 (2008b:18006), Zbl 1164.18300.

\bibitem[Hir03]{hirschhornmodel}{\scshape Hirschhorn, P.S.} Model categories and their localizations. Mathematical Surveys and Monographs., 99. {\em American Mathematical Society, Providence, RI}, 2003. MR1944041 (2003j:18018), Zbl 1017.55001.

\bibitem[Joh75]{johnstoneadjoint}{\scshape Johnstone, P.T.} Adjoint lifting theorems for categories of algebras. {\em Bull. London Math. Soc.} {\bf 7(3)} (1975) 294--297. MR0390018 (52 \#10845), Zbl 0315.18004.

\bibitem[Joy08]{joyaltheoryI}{\scshape Joyal, A.} The theory of quasi-categories I. In progress, 2008.

\bibitem[Kel74]{kellydoctrinal}{\scshape Kelly, G.M.} Doctrinal adjunction, I. Proc. {\em Category Seminar} (Proc. Sem., Sydney, 1972/1973). Lect. Notes in Math., 420, {\em Springer, Berlin}, 1974, 257--280. MR0360749 (50 \#13196), Zbl 0334.18004.

\bibitem[Kel80]{kellyunified}{\scshape Kelly, G.M.} A unified treatment of transfinite constructions for free algebras, free monoids, colimits, associated sheaves, and so on. {\em Bull. Austral. Math. Soc.} {\bf 22(1)} (1980) 1--83. MR589937 (82h:18003), Zbl 0437.18004.

\bibitem[KS74]{kellystreetreview}{\scshape Kelly, G.M.; Street, R.} Review of the elements of 2-categories, I. Proc. {\em Category Seminar} (Proc. Sem., Sydney, 1972/1973). Lect. Notes in Math., 420, {\em Springer, Berlin}, 1974, 75--103. MR0357542 (50 \#10010), Zbl 0334.18016.

\bibitem[KT93]{ktfactorization}{\scshape Korostenski, M.; Tholen, W.} Factorization systems as Eilenberg-Moore algebras. {\em J. Pure Appl. Algebra} {\bf 85(1)} (1993) 57--72. MR1207068 (94a:18002), Zbl 0778.18001.

\bibitem[Lac07]{lackhomotopy}{\scshape Lack, S.} Homotopy-theoretic aspects of 2-monads. {\em J. Homotopy Relat. Struct.} {\bf 2(2)} (2007) 229--260. MR2369168 (2008m:18006), Zbl 1184.18005.

\bibitem[May75]{mayclassifying}{\scshape May, J.P.} Classifying spaces and fibrations. {\em Mem. Amer. Math. Soc.} {\bf 1(155)} (1975) 1--98. MR0370579 (51 \#6806), Zbl 0321.55033.

\bibitem[MP11]{maypontoconcise2}{\scshape May, J.P.; Ponto, K.} More concise algebraic topology. To appear, {\em University of Chicago Press}, 2011.

\bibitem[Nik10]{nikolausalgebraic}{\scshape Nikolaus, T.} Algebraic models for higher categories. Preprint, arXiv:1003.1342v1 [math.AT] 2010.

\bibitem[PW02]{powerwatanabecombining}{\scshape Power, J.; Watanabe, H.} Combining a monad and a comonad. {\em Theoret. Comput. Sci.} {\bf 280(1-2)} (2002) 137--162. MR1905224 (2003g:18002), Zbl 1002.68059.

\bibitem[Qui67]{quillenhomotopical}{\scshape Quillen, D.G.} Homotopical algebra. Lect. Notes in Math., 43, {\em Springer-Verlag, Berlin}, 1967. MR0223432 (36 \#6480), Zbl 0168.20903.

\bibitem[RB99]{radulescubanu}{\scshape Radulescu-Banu, A.} Cofibrance and completion. Ph.D. thesis, Massachusetts Institute of Technology, 1999.

\bibitem[RT02]{rtlax}{\scshape Rosick\'{y}, J.; Tholen, W.} Lax factorization algebras. {\em J. Pure Appl. Algebra} {\bf 175(1-3)} (2002) 355-382. MR1935984 (2003j:18003), Zbl 1013.18001.

 
\end{thebibliography}
\end{document}